\documentclass[11pt]{amsart}%

\usepackage{amsmath, amssymb, amsthm, amsfonts, bbm, dsfont}
\usepackage{graphicx}

\usepackage{mathrsfs}

\usepackage{hyperref}

\usepackage{a4wide}

\numberwithin{equation}{section}
\theoremstyle{plain}
\newtheorem{theorem}[subsubsection]{Theorem}
 
 \newtheorem{proposition}[subsubsection]{Proposition}
 \newtheorem{corollary}[subsubsection]{Corollary}
 \newtheorem{conjecture}[subsubsection]{Conjecture}

 \theoremstyle{definition}

\newtheorem{remark}[subsubsection]{Remark}

\usepackage{pictexwd}
\usepackage{stackengine}
\usepackage{mathtools}
\usepackage{amssymb}
\usepackage{dsfont}

\usepackage{mathabx}

\newcommand{\CC}{\mathbb{C}}

\newcommand{\PP}{\mathbb{P}}

\newcommand{\ZZ}{\mathbb{Z}}

\newcommand{\bfS}{\mathbf{S}}

\newcommand{\calF}{\mathcal{F}}
\newcommand{\calG}{\mathcal{G}}
\newcommand{\calO}{\mathcal{O}}

\newcommand{\calB}{\mathcal{B}}
\newcommand{\calP}{\mathscr{P}}
\newcommand{\calL}{\mathcal{L}}

\newcommand{\frb}{\mathfrak{b}}

\newcommand{\frn}{\mathfrak{n}}

\newcommand{\bclC}{\bar{\mathcal{C}}}

\newcommand{\supp}{\textup{supp}}

\newcommand{\Ind}{\textup{Ind}}
\newcommand{\ind}{\textup{ind}}

\newcommand\Lie{\textup{Lie}\ }

\newcommand\Mod{\textup{Mod}}

\newcommand{\Tr}{\textup{Tr}}

\newcommand\Hom{\textup{Hom}}

\newcommand{\Ad}{\textup{Ad}}

\newcommand{\Id}{\mathrm{Id}}

\newcommand{\quash}[1]{}

\newcommand{\ep}{\epsilon}

\newcommand{\Br}{\mathfrak{Br}}

\usepackage{tikz}
\usetikzlibrary{shapes,fit,intersections}
\usetikzlibrary{decorations.markings}
\usetikzlibrary{decorations.pathreplacing}
\usetikzlibrary{patterns}
\usetikzlibrary{matrix,arrows}
\usepgflibrary{decorations.pathmorphing}

\usepackage{tikz-cd}

\newcommand{\Wr}{\overline{W}}
\newcommand{\Hilb}{\textup{Hilb}}
\newcommand{\bHilb}{\mathbf{\textup{Hilb}}}

\newcommand{\MF}{\mathrm{MF}}
\newcommand{\calXr}{\overline{\mathcal{X}}}
\newcommand{\calX}{\mathcal{X}}
\newcommand{\MFs}{\mathrm{MF}}
\newcommand{\calZ}{\mathcal{Z}}
\newcommand{\calC}{\mathcal{C}}
\newcommand{\calY}{\mathcal{Y}}
\newcommand{\cox}{\mathrm{cox}}

\newcommand{\frh}{\mathfrak{h}}

\newcommand{\CE}{\mathrm{CE}}

\newcommand{\odel}{\stackon{$\otimes$}{$\scriptstyle\Delta$}}
\newcommand{\bst}{\bar{\star}}

\def\IR{ \mathbb{R} }

\def\ZZ{\mathbb{Z} }

\def\MFs{\mathrm{MF}}

\def\xId{ \mathbbm{1}}
\def\xIdv#1{ \xId_{#1}}

\def\kmtr#1{ \begin{bmatrix} #1 \end{bmatrix} }

\def\Zt{ \ZZ_2 }
\def\xD{ D }

\def\xxM{ M }
\def\xxV{ V }
\def\xdl{ d_{\mathrm{l}} }
\def\xdr{ d_{\mathrm{r}} }

\def\Ksz{ \mathrm{K} }
\def\Kszvvv#1#2#3{ \Ksz(#1;#2,#3) }
\def\ktht{\theta}

\def\exbv#1{b_{#1}}

\def\exbtt{\exbv{22}}

\def\exbv#1{b_{#1}}

\def\exbtt{\exbv{22}}

\def\exbhh{ \exbv{33}}

\newenvironment{dedication}
    {\vspace{6ex}\begin{quotation}\begin{center}\begin{em}}
    {\par\end{em}\end{center}\end{quotation}}

\title{HOMFLYPT homology of Coxeter links}
\author{A. Oblomkov}
\address{
A.~Oblomkov\\
Department of Mathematics and Statistics\\
University of Massachusetts at Amherst\\
Lederle Graduate Research Tower\\
710 N. Pleasant Street\\
Amherst, MA 01003 USA
}
\email{oblomkov@math.umass.edu}

\author{L. Rozansky}
\address{
L.~Rozansky\\
Department of Mathematics\\
University of North Carolina at Chapel Hill\\
CB \# 3250, Phillips Hall\\
Chapel Hill, NC 27599 USA
}
\email{rozansky@math.unc.edu}

\begin{document}
\maketitle

 \begin{dedication} Dedicated to the memory of Jim Humphreys.\end{dedication}

\begin{abstract}
A Coxeter link is a closure of a product of two braids, one being a quasi-Coxeter element and the other being a product of partial full twists.
This class of links includes torus knots \(T_{n,k}\) and torus links \(T_{n,nk}\).
We identify the knot homology of a Coxeter link with the space of sections of a particular line bundle on a natural generalization   of the punctual locus inside the flag Hilbert  scheme of points in \(\CC^2\).
 \end{abstract}

\section{Introduction}
\label{sec:introduction}
In the seminal paper \cite{Jones} Jones introduced what would be later called the HOMFLYPT polynomial invariant \(P(L)\) of a link \(L\) in $\IR^3$.
Besides the definition, the paper has many amazing results and computations. In particular, the section 9
of \cite{Jones} contains a proof of a formula for the HOMFLYPT invariant of torus knots \(T_{m,n}\). Later the HOMFLYPT invariant was
upgraded to the homology theory \cite{KhR08a,KhR08b}.
In this paper we demonstrate that the Jones formula has a natural generalization to the homology theory for a special class of torus links.

\def\CCq{ \CC^*_q }
\def\CCt{ \CC^*_t}
\def\CCqt{ \CC^*_{q,t}}

Consider the plane $\CC^2$ with the action of the group $\CC^*$ denoted as $\CCq$:  \(\lambda\cdot(x,y)=(\lambda x,\lambda^{-1}y).\) This action extends to the Hilbert scheme \(\Hilb_n(\CC^2)\) which is a variety of ideals \(I\subset\CC[x,y]\) of
codimension \(n\).
%
%
%
The tautological vector bundle \(\calB\) whose fiber over \(I\) is the vector space dual to  \(\CC[x,y]/I\) is naturally $\CCq$-equivariant. Combining the localization
formula of Atiyah and Bott \cite{AB} with the result of Haiman \cite{Hai} we get an algebro-geometric version of the Jones formula:
%
%
\[P(T_{1+kn,n})=\sum_{i=0}^n\dim_q\left(\mathrm{H}^0\bigl(\Hilb_n(\CC^2),\calO_Z\otimes L^k\otimes\Lambda^i\calB\bigr)\right)a^i,\]
where \(Z\subset\Hilb_n(\CC^2)\) is the punctial Hilbert  scheme  consisting of ideals \(I\) with support  at \((0,0)\in \CC^2\),
and \(\dim_q\) is the dimension graded by $\CCq$-weights.

Many authors \cite{AS12,GN15,GORS12,ORS12} suggested that the Poincare polynomial $ \calP(T_{1+kn,n})$ of the triply graded HOMFLYPT homology~\cite{KhR08a,KhR08b} has a similar interpretation, if one
augments the action of $\CCq$ to that of $T_{sc}=\CCq\times\CCt$, where $\CCt$: \( \mu\cdot(x,y) = (x,\mu^2 y)\), and uses the $\CCqt$-weighted dimension:
%
\begin{equation}\label{eq:torushoms}
  \calP(T_{1+kn,n})=\sum_{i=0}^n\dim_{q,t}\left(\mathrm{H}^0\bigl(\Hilb_n,\calO_Z\otimes L^k\otimes\Lambda^i\calB\bigr)\right)a^i.
\end{equation}



While we were finishing this preprint, M.~Hogankamp published a proof of the conjecture \cite{Ho17,M17}.
He used the construction of
the HOMFLYPT homology via Soergel bimodules and matched combinatorics of the complexes of bimodules that appear in knot homology of torus knots with the combinatorics of
the generalized Catalan numbers, the latter related to the sections of \(L^k\) by a combination of the results \cite{CM,M16,Hai}.

The paper \cite{Ho17} is  a real tour de force in combinatorics and homological algebra, however it does not provide a natural explanation for the appearance of \(\Hilb(\CC^2)\) in knot homology.
When the conjecture (\ref{eq:torushoms}) appeared, the  available constructions for
triply graded homology had no obvious connections with coherent sheaves on this variety.

A direct relation between the triply graded knot homology and $\CCqt$-equivariant coherent sheaves on \(\Hilb(\CC^2)\) was
 established  by the authors \cite{OR16} (see also the paper \cite{GorskyNegutRasmussen16} where a
 K-theoretic version of this relation is suggested).  Recently, it was also shown
 \cite{OblomkovRozansky20}
 by
 the authors that the link homology  from \cite{OR16} coincides with the Khovanov-Rozansky link homology \cite{KhR08a}.

From the papers \cite{OR16}, as well as \cite{GN15,GorskyNegutRasmussen16}, it is clear that the natural home for the algebro-geometric version of the HOMFLYPT homology
is the category of the quasi-coherent sheaves on the {\it nested} Hilbert scheme  \(\Hilb_{1,n}\) parameterizing chains of ideals \(I_1\supset I_2\supset\dots \supset I_n\)
with support of \(I_i/I_{i+1}\) being a point on the line \(y=0\). There is a natural analog \(Z_{1,n}\) of the punctual Hilbert scheme  \(Z\) in the nested case which consists of the
chains of ideals with the support of \(I_i/I_{i+1}\) at \((x,y)=(0,0)\). However,  the natural analogue of \(\calO_Z\) turns to be the Koszul complex of the defining
equations for \(Z_{1,n}\) which we denote by \([\calO_{Z_{1,n}}]^{vir}\) and define in section~\ref{sec:link}.
Finally, the weights of $\CCt$-action are combined with homological degree which means that all differentials have $\CCt$-weight one and the variable $y$ has homological degree two.

The main result of this paper is the following:

\begin{theorem} For any positive \(n,k\) we have
  \[\calP(T_{1+kn,n})=\sum_{i=0}^n\dim_{q,t}\left(\mathrm{H}^*\bigl(\Hilb_{1,n},[\calO_{Z_{1,n}}]^{vir}\otimes L^k\otimes\Lambda^i\calB\bigr)\right)a^i,
\]
where \(\calP\) is the Poincare polynomial for the triply graded homology \footnote{In this paper we use the term the triply graded homology for the homology theory
from \cite{OR16}, it is shown in \cite{OblomkovRozansky20} that the homology from \cite{OR16} are to the triply-graded homology of \cite{KhR08a}. }.
\end{theorem}

This paper is a natural continuation of our  previous papers \cite{OR17,OR16}. In the second paper we prove the relation between the homology of
the closure \(L(\beta)\) of \(\beta\in Br_n\) and of closure of \(\beta\cdot \delta^{\vec{k}}\) where \(\delta^{\vec{k}}:=\prod_{i=1}^n \delta_i^{k_i}\) is the product of
the JM elements \[\delta_i:=\sigma_i\sigma_{i+1}\dots \sigma_{n-1}^2\dots\sigma_{i+1}\sigma_i,\quad i=1,\dots,n-1,\]
here \(\sigma_i\) are the standard generators for the braid group \(\Br_n.\)
The above mentioned formula  for the homology of \(T_{1+kn,n}\) is obtained by applying result of \cite{OR17} for \(\beta=\sigma_1\dots\sigma_{n-1}\) and
\(k_1 =\cdots = k_n = k\).
To apply the result of \cite{OR17} we need to analyze the sheaf-theoretic object that the theory from \cite{OR16} assigns to
the braid \(\beta\) which we call the Coxeter braid.

More generally, we study the sheaf-theoretic object that is attached by the theory from \cite{OR16} to the general quasi-Coxeter braid:
\[\cox_{S}:=\overrightarrow{\prod_{i\notin S}}\sigma_i,\]
where \(S\subset\{1,\dots,n-1\}\) is a subset and the product is taken in the descending order of the indices. In particular, we identify the
homology of the closure of element \(\cox_{S}\cdot\delta^k\) for any \(S\) and \(k\). We call these closures Coxeter links. This is a wide class of links
which includes the torus links \(T_{m,n}\), \((m,n)=1\).
The class also contains the torus link
\(T_{n,kn}\).

The Khovanov-Rozansky  homology of the links \(T_{n,nk}\) and knots \(T_{n,k}\) were studied in \cite{EH16} and in \cite{Ho17},\cite{M17} and would be interesting
to make a connection between our results and technique of these papers.

The nested Hilbert  scheme  \(\Hilb_{1,n}\) carries a natural line bundle \(\calL_i\) whose fiber over \(I_\bullet\) is the quotient \(I_i/I_{i+1}\).
For any subset \(S\subset\{1,\dots,n-1\}\) we define \(Z^S_{1,n}\subset\Hilb_{1,n}\) to be a subscheme defined by the condition \(\supp(I_{i-1}/I_{i})=
\supp(I_{i}/I_{i+1}))\) for all \(i\notin S\). We prove

\begin{theorem}\label{thm:coxlinks} For any \(S\subset\{1,\dots,n-1\}\) and \(\vec{k}\in \ZZ^{n-1}\)
  \[\calP\bigl(L(\cox_S\cdot \delta^{\vec{k}})\bigr)=\sum_{i=0}^n\dim_{q,t}\left(\mathrm{H}^*\bigl(\Hilb_{1,n}, [\calO_{Z_{1,n}^S}]^{vir}\otimes\calL^{\vec{k}}\otimes\Lambda^i\calB\bigr)\right)a^i.
\]
\end{theorem}

We prove this theorem in section~\ref{sec:coxet-matr-fact} and section~\ref{sec:link}. We also provide some short overview of the methods of \cite{OR16} in the section~\ref{sec:matr-fact}.
If the vector \(\vec{k}\) is sufficiently positive, we can use Atiyah-Bott localization \cite{AB} to compute the graded dimensions in this formula similar to the
one from \cite{GorskyNegutRasmussen16}, see theorem~\ref{thm:localization} below.

It turns out that the localization approach only works under some vanishing conditions on the sheaf  homology like in \cite{Hai}. 
In the section~\ref{sec:local-expl-form} we show that the easiest version of the vanishing condition \cite{OblomkovRozansky18a} implies

\begin{theorem}\label{thm:localization}
  For  \(\vec{k}\in \ZZ_{>0}\) such that \(k_1>k_2>\dots>k_{n-1}\) there is \(M\) such that
  we have the following explicit formula for the knot invariant:
  \[\calP\bigl(L( 1\cdot\delta^{\vec{b}})\bigr)=\sum_{p\in \Hilb_{1,n}^{T}} \Omega_p(Q,T,a)\, Q^{\vec{b}\cdot w_x(p)}\, T^{\vec{b}\cdot w_y(p)},\]
  where \(Q=q^2/t^2,T=t^2\), \(\vec{b}=\vec{k}+r\vec{1}\), \(\vec{1}=(1,\dots,1)\), \(r>M\) and the weight \(\Omega\) and vectors \(w_\bullet(p)\) can be explicitly computed and depends only on \(p\).
\end{theorem}

The formulas for \(\Omega_p(Q,T,a)\) and \(w_\bullet(p)\) are given in section~\ref{sec:local-expl-form} and theorem is proved at the end of section~\ref{sec:fourth-grad-local}.
We conjecture that the theorem can be strengthened in two direction: we can replace the Coxeter braid \(1=\cox_{\{1,\dots,n-1\}}\) by any Coxeter braid \(\cox_S\) and
we give a precise criterion for the vector \(\vec{k}\)  to be sufficiently positive.
We formulate these conjectures in section \ref{sec:conjectures} and provide evidence in their support. Note that the weight \(\Omega_p\) appears to be equal to the localization weight from the main formula of
\cite{GN15}.

We do not expect a simple localization formula for the Poincare polynomials of non-positive links. The examples of such links
are discussed
in the section~\ref{sec:expl-comp} and section~\ref{sec:local-expl-form}.

The first author was lucky to have an advice of Jim Hamphreys on multiple subjects of mathematics. In particular, quasi-Coxeter braids were suggested by Jim as a class of braids that tends to lead to more computable objects. This suggestion was a starting  point of this paper.

{\bf Acknowledgments} We would like to thank Dmitry Arinkin, Eugene Gorsky, Roman Bezrukavnikov, Andrei Negu{\c t} for useful discussions. Also we would like thank Eugene Gorsky and Andrei
Negu{\c t} for  a careful reading of the first version of this paper and suggestion that helped to improve the text.
The work of A.O. was supported in part by  the NSF CAREER grant DMS-1352398 and NSF-FRG grant DMS-1760373.
The work of L.R. was supported in part by
the NSF grant DMS-1760578.

\section{Matrix factorizations and knot invariants}
\label{sec:matr-fact}
The construction of link invariants in~\cite{OR16} is based on a homomorphism from the braid group to a special monoidal category of matrix factorizations. The main result of this paper follows from the explicit computation of the images of Coxeter braids.
%

\subsection{Matrix Factorizations}
\label{ssec:matr-fact}
Matrix factorizations were introduced by  Eisenbud \cite{E80} and further developed by Orlov \cite{O04}, see \cite{D11} for a review. Here we present only the basic definitions and omit proofs.

The category of matrix factorizations $\MFs(Z,F)$ is a triangulated category based on an affine variety $Z$ and 
a function $F\in \CC[Z]$.
An object of this category is a $\ZZ_2$-graded free $\CC[Z]$-module $M=M_0\oplus M_1$ of finite rank equipped with
a degree one endomorphism $D$ called a curved differential:
\[ \mathcal{F}=(M_0\oplus M_1 ,D),\quad D: M_i\rightarrow M_{i+1},\quad D^2=F.\]

Given \(\mathcal{F}=(M,D)\) and \(\mathcal{G}=(N,D')\) the linear space of morphisms \(\Hom(\mathcal{F},\mathcal{G})\) consists of the homomorphisms of
\(\CC[Z]\)-modules
\(\phi=\phi_0\oplus \phi_1\), \(\phi_i\in\Hom(M_i,N_i) \) such that \(\phi\circ D=D'\circ \phi\).
Two morphisms \(\phi,\rho\in \Hom(\mathcal{F},\mathcal{G})\) are homotopic if there is homomorphism of \(\CC[Z]\)-modules \(h=h_0\oplus h_1 \),
\(h_i\in \Hom(M_i,N_{i+1})\) such that \(\phi-\rho=D'\circ h-h\circ D\).

In the paper \cite{OR16} we introduced a notion of the equivariant matrix factorizations which we explain below.
First, recall the construction of the Chevalley-Eilenberg complex.

\subsection{Chevalley-Eilenberg complex}
\label{sec:chev-eilenb-compl}

Suppose that $\frh$ is a Lie algebra. Chevalley-Eilenberg complex
 $\CE_\frh$ is the complex $(V_\bullet(\frh),d)$ with $V_p(\frh)=U(\frh)\otimes_\CC\Lambda^p \frh$ and differential $d_{ce}=d_1+d_2$ where:
 \def\dtheta{d}
 $$ d_1(u\otimes x_1\wedge\dots \wedge x_p)=\sum_{i=1}^p (-1)^{i+1} ux_i\otimes x_1\wedge\dots \wedge \hat{x}_i\wedge\dots\wedge x_p,$$
 $$ d_2(u\otimes x_1\wedge\dots \wedge x_p)=\sum_{i<j} (-1)^{i+j} u\otimes [x_i,x_j]\wedge x_1\wedge\dots \wedge \hat{x}_i\wedge\dots\wedge \hat{x}_j\wedge\dots \wedge x_p,$$

 Let us define by $\Delta$ the standard map $\frh\to \frh\otimes \frh$ defined by $x\mapsto x\otimes 1+1\otimes x$.
 Suppose $V$ and $W$ are modules over Lie algebra $\frh$ then we use notation
 $V\odel W$ for $\frh$-module which is isomorphic to $V\otimes W$ as vector space, the $\frh$-module structure being defined by  $\Delta$. Respectively, for a given $\frh$-equivariant matrix factorization $\calF=(M,D)$ we denote by $\CE_{\frh}\odel \calF$
 the $\frh$-equivariant matrix factorization $(CE_\frh\odel\calF, D+d_{ce})$. The $\frh$-equivariant structure on $\CE_{\frh}\odel \calF$ originates from the
 left action of $U(\frh)$ that commutes with right action on $U(\frh)$ used in the construction of $\CE_\frh$.

 A slight modification of the standard fact that $\CE_\frh$ is the resolution of the trivial module implies that \(\CE_\frh\stackon{$\otimes$}{$\scriptstyle\Delta$} M\) is a free resolution of the
$\frh$-module $M$.

\subsection{Equivariant matrix factorizations}
\label{sec:equiv-matr-fact}

Let us assume that there is an action of the Lie algebra \(\frh\) on \(\calZ\) and \(F\) is a \(\frh\)-invariant function.
Then we can construct the following triangulated category \(\MFs_{\frh}(\calZ,W)\).

The objects of the category are the triples:
\[\mathcal{F}=(M,D,\partial),\quad (M,D)\in\MFs(\calZ,W) \]
where $M=M^0\oplus M^1$ and $M^i=\CC[\calZ]\otimes V^i$, $V^i \in \Mod_{H}$,
$\partial\in \oplus_{i>j} \Hom_{\CC[\calZ]}(\Lambda^i\frh\otimes M, \Lambda^j\frh\otimes M)$ and $D$ is an odd endomorphism
$D\in \Hom_{\CC[\calZ]}(M,M)$ such that
$$D^2=F,\quad  D_{tot}^2=F,\quad D_{tot}=D+d_{ce}+\partial,$$
where the total differential $D_{tot}$ is an endomorphism of $\CE_\frh\odel M$, that commutes with the $U(\frh)$-action.
The morphism \(\partial\) is called {\it correction differential}.


Note that we do not impose the equivariance condition on the differential $D$ in our definition of matrix factorizations. On the other hand, if $\calF=(M,D)\in \MFs(\calZ,F)$ is a matrix factorization with
$D$ that commutes with $\frh$-action on $M$ then $(M,D,0)\in \MFs_\frh(\calZ,F)$.
We call such matrix factorization {\it strictly equivariant}.


Given two $\frh$-equivariant matrix factorizations $\calF=(M,D,\partial)$ and $\tilde{\calF}=(\tilde{M},\tilde{D},\tilde{\partial})$ the space of morphisms $\Hom(\calF,\tilde{\calF})$ consists of
homotopy equivalence classes of elements $\Psi\in \Hom_{\CC[\calZ]^\frh}(\Lambda^\bullet \frh\otimes M, \Lambda^\bullet\frh\otimes \tilde{M})$ such that $\Psi\circ D_{tot}=\tilde{D}_{tot}\circ \Psi$ and $\Psi$ commutes with
$U(\frh)$-action on $\CE_\frh\odel M$. Two map $\Psi,\Psi'\in \Hom(\calF,\tilde{\calF})$ are homotopy equivalent if
there is \[ h\in  \Hom_{\CC[\calZ]}(\CE_\frh\odel M,\CE\frh\odel\tilde{M})\] such that $\Psi-\Psi'=\tilde{D}_{tot}\circ h+ h\circ D_{tot}$ and $h$ commutes with $U(h)$-action on  $\CE_\frh\odel M$.

 Given two $\frh$-equivariant matrix factorizations $\calF=(M,D,\partial)\in \MFs_\frh(\calZ,F)$ and $\tilde{\calF}=(\tilde{M},\tilde{D},\tilde{\partial})\in \MFs_\frh(\calZ,\tilde{F})$
 $\calF\otimes\tilde{\calF}\in \MFs_\frh(\calZ,F+\tilde{F})$ as the equivariant matrix factorization $(M\otimes \tilde{M},D+\tilde{D},\partial+\tilde{\partial})$.

 \subsection{Push forwards, quotient by the group action}
\label{sec:push-forwards}

The technical part of \cite{OR16} is the construction of push-forwards of equivariant matrix factorizations. Here we state the main
results, the details may be found in section 3 of \cite{OR16}. We need push forwards along projections and embeddings. We also use  the
functor of taking quotient by group action for our definition of the convolution algebra.

The projection case is more elementary. Suppose \(\calZ=\mathcal{X}\times\mathcal{Y}\), both \(\calZ \) and \(\mathcal{X}\) have \(\frh\)-action and
the projection \(\pi:\mathcal{Z}\rightarrow\mathcal{X}\) is \(\frh\)-equivariant. Then
for any $\frh$ invariant element $w\in\CC[\calX]^\frh$ there is a functor
\(\pi_{*}\colon \MFs_{\frh}(\calZ, \pi^*(w))\rightarrow \MFs_{\frh}(\mathcal{X},w)
\)
which simply forgets the action of $\CC[\calY]$.


We define an embedding-related push-forward in the case when the subvariety $\calZ_0\xhookrightarrow{j}\calZ$
is the common zero of an ideal $I=(f_1,\dots,f_n)$ such that the functions $f_i\in\CC[\calZ]$ form a regular sequence. We assume that the Lie algebra $\frh$ acts on $\calZ$ and $I$ is $\frh$-invariant. Then there exists an $\frh$-equivariant Koszul complex $K(I)=(\Lambda^\bullet \CC^n\otimes \CC[\calZ],d_K)$ over $\CC[\calZ]$ which has non-trivial homology only in degree zero. Then in section~3 of \cite{OR16} we define the push-forward functor
\[
j_*\colon \MFs_{\frh}(\calZ_0,W|_{\calZ_0})\longrightarrow
\MFs_{\frh}(\calZ,W),
\]
for any $\frh$-invariant element $W\in\CC[\calZ]^\frh$.


Finally, let us discuss the quotient map. The complex \(\CE_\frh\) is a resolution of the trivial \(\frh\)-module by free modules. Thus the correct derived
version of taking \(\frh\)-invariant part of the matrix factorization \(\mathcal{F}=(M,D,\partial)\in\MFs_\frh(\calZ,W)\), \(W\in\CC[\calZ]^\frh\) is
\[\CE_\frh(\mathcal{F}):=(\CE_\frh(M),D+d_{ce}+\partial)\in\MFs(\calZ/H,W),\]
where \(\calZ/H:=\mathrm{Spec}(\CC[\calZ]^\frh )\) and use the general definition of \(\frh\)-module \(V\):
\[\CE_\frh(V):=\Hom_\frh(\CE_\frh,\CE_\frh\odel V).\]

\subsection{Convolutions on  reduced spaces}
\label{sec:conv-new}
For a Borel group $B$, we treat $B$-modules as \(T\)-equivariant \(\frn=\mathrm{Lie}([B,B])\)-modules. For a space $\calZ$ with $B$-action and for $W\in\CC[\calZ]^B$ we define $\MFs_B(\calZ,W)$ as a full subcategory of $\MFs_{\frn}(\calZ,W)$ whose objects are matrix factorizations \((M, D, \partial)\), where $M$ is a \(B\)-module and the differentials $D$ and $\partial$ are \(T\)-invariant. The category
\(\MFs_{B^\ell}(\calZ,W)\) has a similar definition.


The backbone  of the constructions of the knot invariant from \cite{OR17} is the study of the category of matrix factorizations on the spaces \(\calXr_\ell\):
\(\calXr_\ell:=\frb\times G^{\ell-1}\times\frn\) with the \(B^\ell\)-action:
\[(b_1,\dots,b_\ell)\cdot(X,g_1,\dots,g_{\ell-1},Y)=(\Ad_{b_1}(X),b_1g_1b_2^{-1},b_2g_2b_3^{-1},\dots,\Ad_{b_\ell}(Y)).\]

The categories that we use in \cite{OR16} are subcategories \(\MF_{B^\ell}(\calXr_\ell,F)\subset \MFs_{B^\ell}(\calXr_\ell,F)\) that consist of the matrix factorizations which
are equivariant with respect to the action of \(T_{sc}\) and \(G\)-invariant. In particular the space \(\calXr_2\) has the following \(B^2\)-invariant potential:
\[\Wr(X,g,Y)=\Tr(X\Ad_g(Y)),\]
and the category \(\MF_{B^\ell}(\calXr_\ell,\Wr)\) has a structure of the convolution algebra \cite{OR16} that we outline below.

There are the following
maps $\bar{\pi}_{ij}:\calXr_3\to\calXr_2$:
\[\bar{\pi}_{12}(X,g_{12},g_{13},Y)=(X,g_{12},\Ad_{g_{23}}(Y)_{++}),
  \quad\bar{\pi}_{13}(X,g_{12},g_{13},Y)=(X,g_{12}g_{23},Y),\]
 \[\bar{\pi}_{23}(X,g_{12},g_{13},Y)=(\Ad_{g_{12}}^{-1}(X)_+,g_{23},Y).\]
Here and everywhere below \(X_+\) and \(X_{++}\) stand for the upper and strictly-upper triangular parts of \(X\).
The maps \(\bar{\pi}_{12}\times\bar{\pi}_{23}\) is  \(B^2\)-equivariant  but not \(B^3\)-equivariant. However in section 5.4 of \cite{OR16} we show that
for any \(\mathcal{F},\mathcal{G}\in\MF_{B^2}(\calXr,\Wr)\) there is a natural element
\begin{equation}\label{eq:conv-red}
(\bar{\pi}_{12}\otimes_B\bar{\pi}_{23})^*(\mathcal{F}\boxtimes\mathcal{G})\in\MF_{B^3}(\calXr_3,\bar{\pi}_{13}^*(W)),
\end{equation}

such that  we can define the binary operation on \(\MF_{B^2}(\calXr,\Wr)\):
\[\mathcal{F}\bar{\star}\mathcal{G}:=\bar{\pi}_{13*}(\CE_{\frn^{(2)}}((\bar{\pi}_{12}\otimes_B\bar{\pi}_{23})^*(\mathcal{F}\boxtimes\mathcal{G}))^{T^{(2)}}).\]

Instead of going into details of the construction of the convolution algebra let us explain the induction functors \cite{OR16} that provide us with an effective
method of computing of the convolution product.

\subsection{Induction functors}
\label{sec:induction-functors}
The standard parabolic subgroup \(P_k\) has Lie algebra generated by \(\frb\) and \(E_{i+1,i}\), \(i\ne k\).
Let us define space  \(\calXr_2(P_k):=\frb\times P_k \times \frn\) and let us also use notation
\(\calXr_2(G_n)\) for \(\calXr_2(G_n)\). There is a natural embedding \(\bar{i}_k:\calXr_2(P_k)\rightarrow\calXr_2\) and
a natural projection \(\bar{p}_k:\calXr_2(P_k)\rightarrow\calXr_2(G_k)\times\calXr_2(G_{n-k})\). The embedding \(\bar{i}_k\) satisfies
the conditions for existence of the push-forward and we can define the induction functor:
\[\overline{\ind}_k:=\bar{i}_{k*}\circ \bar{p}_k^*: \MF_{B_k^2}(\calXr_2(G_k),\Wr)\times\MF_{B_{n-k}^2}(\calXr_2(G_{n-k}),\Wr)\rightarrow\MF_{B_n^2}(\calXr_2(G_{n}),\Wr)\]

Similarly we define space  \(\calXr_{2,fr}(P_k)\subset\frb\times P_k \times \frn\times V\) as an open subset defined by the stability condition:
\begin{equation}
  \CC\langle X,\Ad_g^{-1}(Y)\rangle u=V, \quad g^{-1}(u)\in V^0.
  \end{equation}

The last space has a natural projection map \(\bar{p}_k:\calXr_{2,fr}(P_k)\rightarrow\calXr_2(G_k)\times\calXr_{2,fr}(G_{n-k})\) and
the embedding \(\bar{i}_k: \calXr_{2,fr}(P_k)\rightarrow\calXr_{2,fr}(G_{n})\) and we can define the induction functor:
\[\overline{\ind}_k:=\bar{i}_{k*}\circ \bar{p}_k^*: \MF_{B_k^2}(\calXr_2(G_k),\Wr)\times\MF_{B_{n-k}^2}(\calXr_{2,fr}(G_{n-k}),\Wr)\rightarrow\MF_{B_n^2}(\calXr_{2,fr}(G_{n}),\Wr)\]

It is shown in section 6 (proposition 6.2) of \cite{OR16} that the functor \(\overline{\ind_k}\) is the homomorphism of the convolution algebras:
\[\overline{\ind}_k(\mathcal{F}_1\boxtimes\mathcal{F}_2)\bar{\star} \overline{\ind}_k(\mathcal{G}_1\boxtimes\mathcal{G}_2)=
  \overline{\ind}_k(\mathcal{F}_1\bar{\star}\mathcal{G}_2\boxtimes\mathcal{F}_2\bar{\star}\mathcal{G}_2).\]

Let us define \(B^2\)-equivariant embedding \(i: \calXr_2(B_n)\rightarrow\calXr_2\), \(\calXr_2(B):=\frb\times B\times \frn\).
The pull-back of \(\Wr\) along the map \(i\) vanishes and the embedding \(i\) satisfies the conditions for existence of the push-forward
\(i_*:\MF_{B^2}(\calXr_2(B_n),0)\rightarrow \MF_{B^2}(\calXr_2(G_n),\Wr)\). We denote by \(\underline{\CC[\calXr_2(B_n)]}\in\MF_{B^2}(\calXr_2(B_n),0)\) the
matrix factorization with zero differential that is homological only in even homological degree. As it is shown in proposition 7.1 of \cite{OR16} the
push-forward
\[\bar{\mathds{1}}_n:=i_{*}(\underline{\CC[\calXr_2(B_n)]})\]
is the unit in the convolution algebra. Similarly, \(\mathds{1}_n:=\Phi(\bar{\mathds{1}}_n)\) is also a unit in non-reduced case.

Using the induction functor and the unit in the convolution algebra
we define  the insertion functor that inserts matrix factorization of smaller rank inside the higher rank one:
\[\overline{\Ind}_{k,k+1}:\MF_{B_2^2}(\calXr_2(G_2),\Wr)\rightarrow\MF_{B_n^2}(\calXr_2(G_n),\Wr)\]
\[\overline{\Ind}_{k,k+1}(\calF):=\overline{\ind}_{k+1}(\overline{\ind}_{k-1}(\bar{\mathds{1}}_{k-1}\times \calF)\times\bar{\mathds{1}}_{n-k-1}),\]

\subsection{Generators of the braid group}
\label{sec:gener-braid-group}

Let us first discuss the case of the braids on two strands. The key to construction of the braid group action in \cite{OR16} is the following factorization in the case
\(n=2\):
$$\Wr(X,g,Y)=y_{12}(2g_{11}x_{11}+g_{21}x_{12})g_{21}/\det,$$
where \(\det=\det(g)\) and
$$ g=\begin{bmatrix} g_{11}&g_{12}\\ g_{21}& g_{22}\end{bmatrix},\quad X=\begin{bmatrix} x_{11}&x_{12}\\ 0& x_{22}\end{bmatrix},\quad Y=\begin{bmatrix} 0& y_{12}\\0&0\end{bmatrix}$$
Thus we can define the following strongly equivariant Koszul matrix factorization:
\[\bar{\mathcal{C}}_+:=(\CC[\calXr_2]\otimes \Lambda\langle\theta\rangle,D,0,0)\in\MF_{B^2}(\calXr_2,\Wr),\]
\[  \quad D=\frac{g_{12}y_{12}}{\det}\theta+g_{11}(x_{11}-x_{22})+g_{21}x_{12}\frac{\partial}{\partial\theta},\]
where \(\Lambda\langle\theta\rangle\)   is the exterior algebra with one generator.

This matrix factorization corresponds to the positive elementary braid on two strands.

Using the insertion  functor we can extend the previous definition on the case of the arbitrary number of strands:
\[\bar{\calC}_+^{(k)}:=\overline{\Ind}_{k,k+1}(\bar{\calC}_+).\]

The section 11 of \cite{OR16} is devoted to the proof of the braid relations between these elements:
\[\bar{\calC}^{(k+1)}_+\bar{\star}\bar{\calC}^{(k)}_+\bar{\star}\bar{\calC}^{(k+1)}_+=\bar{\calC}^{(k)}_+\bar{\star}\bar{\calC}_+^{(k+1)}
  \bar{\star}\bar{\calC}_+^{(k)},\]

Let us now discuss the inversion of the elementary braid. In view of inductive definition of the braid group action, it is sufficient
to understand the inversion in the case \(n=2\).

Thus we define:
\[\bar{\calC}_-:=\bar{\calC}_+\langle-\chi_1,\chi_2\rangle\in\MF_{B^2}(\calXr_2(G_2),\Wr),\]
respectively we define \(\bar{\calC}_-^{(k)}:=\overline{\Ind}_{k,k+1}(\bar{\calC}_-)\).
It is shown in the section 9 of \cite{OR16} that \(\bar{\calC}^{(k)}_-\) is inverse to \(\bar{\calC}^{(k)}_+\).

\subsection{Koszul Matrix Factorizations}
\label{sec:kosz-matr-fact}
The generators of the braid group from the previous subsection are examples of the Koszul matrix factorizations/
Let us remind a  general definition of Koszul matrix factorizations and elementary transformation of the Koszul matrix factorizations. More details on Koszul matrix factorizations in
the
form relevant to the current paper could be found in \cite{OR16}.

Suppose $\calZ$ is a variety with the action of a group $G$ and $F$ is a $G$-invariant potential.
An object of the category $\MFs_{B^2}^{str}(\calZ,F)$ is a free $B^2$-equivariant $\Zt$-graded $\CC[\calZ]$-module $\xxM$ with the odd $G$-invariant differential $\xD$ such that $\xD^2 = F\xIdv{M}$. In particular, a free $G$-equivariant $\CC[\calZ]$-module $\xxV$ with two elements $\xdl\in\xxV$, $\xdr\in\xxV^*$ such that $(v,w)=F$, determines a Koszul matrix factorization $\Kszvvv{\xxV}{\xdl}{\xdr} = \bigwedge^{\bullet}\xxV$ with the differential $Dv = \xdl\wedge v + \xdr\cdot v$ for $v\in \bigwedge^{\bullet}\xxV$. We use a more detailed notation by choosing a basis $\ktht_1,\ldots,\ktht_n\in \xxV$ and presenting $\xdl$ and $\xdr$ in terms of components: $\xdl=a_1\ktht_1 +\cdots a_n\ktht_n$, $\xdr=b_1\ktht^*_1+\cdots + b_n\ktht^*_n$:
\begin{equation}
\label{eq:kszm}
\Kszvvv{\xxV}{\xdl}{\xdr} = \kmtr{a_1 & b_1 &\ktht_1\\ \vdots & \vdots & \vdots \\
a_n & b_n & \ktht_n}
\end{equation}
The structure of $G$-module is described by specifying the action of $G$ on the basis $\ktht_1,\ldots,\ktht_n$.
In some cases when $G$-equivariant structure of the module $\xxM$ is clear from the context we omit the last columns from the notations. We call a matrix presenting Koszul matrix factorization Koszul matrix.
For example, %
%
%
if  we change the basis $\theta_1,\dots, \theta_n$ to the basis $\theta_1,\dots,\theta_i+c\theta_j,\dots, \theta_j,\dots,\theta_n$ the $i$-th and
$j$-th rows of the Koszul matrix will change:
$$ \begin{bmatrix}
a_i& b_i&\theta_i\\
a_j& b_j&\theta_j
\end{bmatrix}
\mapsto
\begin{bmatrix}
a_i+ca_j& b_i&\theta_i+c\theta_j\\
a_j& b_j-cb_i&\theta_j
\end{bmatrix}
$$

Suppose $a_1,\dots,a_n\in \CC[\calZ]$ is a regular sequence and $F\in (a_1,\dots,a_n)$. We can choose $b_i$ such that $F=\sum_i a_i b_i$ and
$\xdl$ and $\xdr$ are as above: $\xdl=a_1\ktht_1 +\cdots a_n\ktht_n$, $\xdr=b_1\ktht^*_1+\cdots + b_n\ktht^*_n$.
In general, there is no unique choice for $b_i$ but all choices lead to homotopy equivalent
Koszul matrix factorizations (in the non-equivariant case they would be simply isomorphic).
In other words, if $b'_i$ is a another collection of elements such that $F=\sum a_i b'_i$  and $\xdr'=b'_1\ktht^*_1+\cdots + b'_n\ktht^*_n$
then Lemma 2.2 from \cite{OR16} imply that the complexes $\Kszvvv{\xxV}{\xdl}{\xdr}$ and $\Kszvvv{\xxV}{\xdl}{\xdr'}$ are homotopy equivalent.
Thus from now on we use notation $\mathrm{K}^F(a_1,\dots,a_n)$ for such matrix factorization.

\section{Coxeter matrix factorization}
\label{sec:coxet-matr-fact}
In the previous section we outlined the definition of the convolution algebra on the category of matrix factorizations. In particular we
explained that for any element \(\beta\in\Br_n\) we can associate a  matrix factorization
\[\bar{\calC}_\beta:=\bclC_{\ep_1}^{(k_1)}\bst\dots\bst\bclC_{\ep_l}^{(k_l)},\]
where \(\beta=\sigma_{k_1}^{\ep_1}\cdot\dots\cdot\sigma_{k_l}^{\ep_l}\) is an expression for \(\beta\) is terms of elementary braids.

We could not expect a simple formula for \(\bclC_\beta \) for a general element \(\beta\in\Br_n\). In particular, as one can see from
computations in the section 11 of \cite{OR16} the matrix factorizations \(\bclC_\beta\) is not always Koszul. Thus it is a bit surprising,
at least for us, that for the Coxeter braid matrix factorization \(\bclC_\beta\) is Koszul and quite simple. To describe the answer we need coordinates on
the space \(\calXr_n=\frb_n\times G_n\times \frn_n\):
\[X=(x_{ij})_{i\leq j},\quad g=(g_{ij}),\quad Y=(y_{ij})_{i\leq j}.\]

Let us introduce \(i\times i\) matrix \(M_i:=[g_{\bullet,1}^{(i)},\dots,g_{\bullet,i-1}^{(i)},\hat{X}^{(i)}_{\bullet,i}]\) where
\(v^{(i)}\) is an abbreviation for vector consisting of first \(i\) entries of \(v\in \CC^n\) and \(\hat{X}=X-x_{11}\Id_n\).
Respectively, we define functions \(F_i:=\det(M_{i+1})\in \CC[\frb\times G]\) and \(I_F\subset \CC[\calXr_2]\) is the ideal generated
by these functions \(F_i\), \(i=1,\dots,n-1\).

\begin{proposition} Let \(I_g=(\{g_{ij}\}_{i-j>1})\) be an ideal in \(\CC[\calXr_2]\), then the ideal \(I_{cox}:=I_g+I_F\) contains \(\Wr\)
\end{proposition}
\begin{proof}
  We show below that the above equations imply that
  \[\Ad^{-1}_g(X)\in \frb.\]
  Indeed, the ideal \(I_g\) defines the sublocus \(Hess\) of Hessenberg matrices of \(G_n\). On the other hand if \(g\in Hess\) then the condition \(F_i=0\) implies that
  the column \(\hat{X}_{\bullet,i+1}\) is a linear combination of the columns \(g_{\bullet,1},\dots,g_{\bullet,i}\). Let us denote by \(K\) the matrix of these coefficients.
  Then we have \(K\) is strictly upper-triangular and \(\hat{X}=g\cdot K\).

  Hence, \(\Ad_{g}^{-1}(\hat{X})=K\cdot g\) but the product of the Hessenberg matrix and strictly upper-triangular matrix is upper-triangular.
\end{proof}

\begin{proposition}
  The functions \(\{g_{ij}\}_{i-j>1},F_1,\dots,F_{n-1}\) form a regular sequence in \(\CC[\frb\times G]\)
\end{proposition}
\begin{proof}
  We proceed by induction. We assume that \(\{g_{ij}\}_{i-j>1},F_1,\dots,F_{n-2}\) form a regular sequence. Then we observe that \(G_n\) is covered by the open sets
  \(U_i\) defined by \(g_i=\det(\Delta_{in})\ne 0\) where \(\Delta_{in}\) is the minor of \(M_n\) obtained by removal of \(in\) entry.
  It is enough to show regularity at every open chart. But in the chart \(U_i\) we have \(F_n/g_{i}=\hat{X}_{in}+\dots\) and \(F_i\), \(i<n\) do not depend on
  \(\hat{X}_{\bullet,n}\). Hence the regularity follows.
\end{proof}

  Thus we can apply Lemma 2.2 from \cite{OR16} to imply that there is a unique
  up to homotopy Koszul matrix factorization \(\mathrm{K}^{\Wr}(\{g_{ij}\}_{i-j>1},F_1,\dots,F_{n-1})\) and we show the following

\begin{theorem} \label{thm:main}  There is a strictly equivariant Koszul  matrix factorization that realizes \(\bclC_{cox}\):
  \[\bclC_{cox}=\mathrm{K}^{\Wr}(\{g_{ij}\}_{i-j>1},F_1,\dots,F_{n-1}).\]
\end{theorem}

The construction of the induction functors implies the following

\begin{corollary} For any \(S\subset\{1,\dots,n-1\}\) we have a strictly equivariant matrix factorization:
  \[\bclC_{cox_S}=\mathrm{K}^{\Wr}(\{g_{ij}\}_{ij\in S'},\{F_i\}_{i\notin S}),\]
  where \(S'=\{ij\}_{i-j>1}\cup \{ii+1\}_{i\in S}\).
\end{corollary}

Before we proceed to  the proof let us describe the most efficient method of computation of the matrix factorization corresponding to the braid \(\beta=\alpha\cdot\sigma_k^\ep\) from
already known \(\calC_\alpha\). The justification of the construction is given in the section 8 of \cite{OR16}.

Indeed, let \(\calF=(M,D_1,\partial_l,\partial_r)\in \MF(\calXr_2(G_n),\Wr)\) where
\(\partial_l,\partial_r\in \Hom_{\CC[\calXr_2(G_n)]}(\Lambda^*\frn\otimes M,M) \) are correcting differential for the equivariant structure. Respectively, \(\bclC_\ep=(R^2_\ep,D_2,0,0)\in\MF_{B_2^2}(\calXr(G_2),\Wr)\)
is strongly equivariant matrix factorization corresponding to elementary braid \(\sigma_k^\ep\), here \(R_\ep\) is the ring \(\CC[\calXr_2(G_2)]\) with appropriately twisted \(B_2^2\)-structure.

The auxiliary space \(\calXr_3(G_n,G_{k,k+1}):=\frb\times G_n\times G_2\times \frn\) is naturally embedded into the convolution space \(\calXr_3\) via map
\(i_{k,k+1}:=\Id^2\times i'_{k,k+1}\times \Id\) where \(j'_{k,k+1}: G_2\rightarrow G_n\) is the embedding of \(G_2\) as \(2\times 2\)-block with entries at positions
\(ij, i,j\in\{k,k+1\}\). Hence we can restrict the maps \(\bar{\pi}_{ij}\) on the auxiliary space, we can also
endow the auxiliary space  with \(B_n\times B_2\times B_n\)-equivariant structure  by restriction from the large space.

The maps \(\bar{\pi}_{ij}\) are \(B_n^2\)-equivariant but not \(B_2\)-equivariant, thus a priori the tensor product
\(\bar{\pi}_{12}^*(\calF)\otimes\bar{\pi}_{23}^*(\calC_\ep)\) has only \(B_n^2\)-equivariant structure. But as
explained in section 8 of \cite{OR16} there is a natural  \(B_n\times B_2\times B_n\)-equivariant matrix factorization \(\calG\) that could be imposed on
\(\bar{\pi}_{12}^*(\calF)\otimes\bar{\pi}_{23}^*(\calC_\ep)\):
\[\calG:=(\bar{\pi}_{12}^*(M)\otimes\bar{\pi}_{23}^*(R_\ep),\bar{\pi}_{12}^*(D_1)+\bar{\pi}_{23}^*(D_2);\partial_l,\partial'_r+\partial',0)\]
where \(\partial'_r\in \Hom_{R_3}(\frn_2\otimes M',M')\), \(M'=M\oplus M=\bar{\pi}_{12}^*(M)\otimes\bar{\pi}_{23}^*(R_\ep)\)
\(R_3=\CC[\calXr_3(G_n,G_{k,k+1})]\) is the restriction of map \(\partial_r\) on the subalgebra  \(\frn_2\)
and \(\partial'\in \Hom_{R_3}(\frn_2\otimes M',M')\) is defined by the formula:
\[\partial':=\bar{\pi}_{12}^*(\frac{\partial D_1}{\partial \tilde{Y}_{k,k+1}})(\tilde{Y}^2_{k+1,k+1}-\tilde{Y}^2_{kk})+\bar{\pi}_{23}\left(\frac{\partial D_2}{\partial \tilde{X}_{kk}}-\frac{\partial D_2}{\partial \tilde{X}_{k+1,k+1}}\right),\]
where \(\tilde{X}^2=\Ad_{g_{12}}(X)\), \(\tilde{Y}^2:=\Ad_{g_{12}}(Y)\) and \(X,g_{12},g_{23},Y\) are the coordinates on \(\calXr_3(G_n,G_{k,k+1})\).

The key observation about this matrix factorization is that up to homotopy we have (see section 8 of \cite{OR16})
\[\bclC_\beta=\bar{\pi}_{13*}(\CE_{\frn_2}(\calG)^{T_2}).\]
Thus we reduce the complexity of the computation of matrix factorization \(\bclC_\beta\), we only need to analyze rank one the Chevalley-Eilenberg complex for \(\frn_2\) and we use this
method in our proof

\begin{proof}[Proof of theorem~\ref{thm:main}]
  Let us first notice that the case \(n=2\) of the theorem is a tautology. The case \(n=3\) was proven in \cite{OR16} in the section 10. For general \(n\) our inductive argument is
  essentially identical to the computation from the section 10 of \cite{OR16}.

  Let \(\alpha=\cox_{n-1}\) then by the induction and the corollary we have a presentation of \(\bclC_\alpha\) as a strongly equivariant Koszul matrix factorization. As induction step we  need to analyze the equivariant matrix factorization \(C_{12}:=\bar{\pi}_{12}^*(\bclC_\alpha)\otimes\bar{\pi}_{23}^*(\bclC_+)\) (with the appropriate \(B_2\)-equivariant structure) on the auxiliary space \(\calXr_3(G_n,G_{n-1,n})\).

 We introduce coordinates on our auxiliary space \(\calXr_3(G_n,G_{n-1,n})\) as follows: \(\calXr_3(G_n,G_{n-1,n})=\{(X,g_{12},g_{23},Y)\}\) where \(g_{12}=(a_{ij})_{i,j\in [1,n]}\)
  \(g_{23}=(b_{ij})_{i,j\in [1,n]}\) and \(b_{ij}=\delta_{ij}\) if \(i,j<n-1\).  Let us also fix notation \(g_{13}:=g_{12}g_{23}=(c_{ij})_{i,j\in[1,n]}\) and
  \[\tilde{X}_2:=\Ad^{-1}_{g_{12}}(X)=(\tilde{x}_{ij})_{i,j\in [1,n]},\quad \tilde{Y}_2=\Ad_{g_{23}}(Y)=(\tilde{y})_{i,j\in [1,n]}.\]
We also use shorthand notations  \(\Delta_a=\det(a),\Delta_b=\det(b),\Delta_c=\det(c).\)

  The matrix factorization \(C_{12}\) is the Koszul matrix factorization \[C_{12}=\mathrm{K}^{\bar{\pi}_{13}^*(\Wr)}\left(\{a_{ij}\}_{ij\in S'},\{F_i(X,a)\}_{i\in[1,n-2]},\tilde{f}\right)\] where
  \(\tilde{f}:=(\tilde{x}_{n-1,n-1}-\tilde{x}_{nn})b_{n-1,n-1}+\tilde{x}_{n-1,n}b_{n,n-1}\) and \(S'=\{i-j>1\}\cup \{(n,n-1)\}\).
  Next let us notice that \(a_{\bullet,i}=c_{\bullet,i}\), \(i\leq n-2\) and since \(F_i(X,a)\) depends
  on \(a_{\bullet,j}\), \(j\leq i\) we obtain a presentation of the complex \(C_{12}\) as a tensor product
  \[ \mathrm{K}^{W'}\left(\{c_{ij}\}_{i-j>1},\{F_i(X,c)_{i\in [1,n-2]}\}\right)\otimes \mathrm{K}^{W''}(a_{n,n-1},\tilde{f}),\]
  where \(W'+W''=\bar{\pi}_{13}^*(\Wr)\) and we can assume that \(W'\) only depends on \(c,X\) but not on \(b\).

  Let's denote the first term in the product
  by \(C'_{12}\) and the second term by \(C''_{12}\). The complex \(C'_{12}\) is \(\frn_2\)-invariant thus \(\CE_{\frn_2}(C_{12})=C'_{12}\otimes \CE_{\frn_2}(C''_{12})\) and to
  complete our proof we need to analyze the last complex in the product. In particular we need an understand \(\frn_2\)-equivariant structure  of \(C''_{12}\)-complex.

  Let \(h\) be an element of \(B_2\subseteq B_n\), that is \(h_{ij}=0\) if \(i,j<n-1\) and \(i\ne j \) or
  \(ij=n,n-1\). The action of \(h\) on the space \(\calXr_2(G_n,G_{n-1,n})\) is given by the formulas:
  \[g_{12}\mapsto g_{12}h^{-1},\quad g_{23}\mapsto hg_{23},\quad \tilde{X}_2\mapsto \Ad_h\tilde{X}_2.\]
We denote by \(\delta\) the element of \(Lie(B_2)\) corresponding to the \((n-1,n)\)-entry and below we investigate its action on the complex \(C''_{12}\).

First, let us notice that the function \(a_{n,n-1}\) is \(\frn_2\)-invariant but the function \(\tilde{f}\) is not. Thus the complex \(C''_{12}\) is not strongly \(\frn\)-equivariant
and correction differentials will appear. In more details we have
\[C''_{12}=
  \begin{bmatrix}
    a_{n,n-1}&*&\theta_1\\\tilde{f}&*&\theta_2
  \end{bmatrix}
,\]
where the action of \(\frn\) is given by:
\[\delta(\theta_1)=k\theta_2,\quad\delta(\theta_2)=0 \]
for some function \(k\in \CC[\frb\times G]\) which we need to compute.

One way to approach the computation of \(k\) is use differentials of \(\frn_2\)-equivariant structure  on \(C_{12}\) from the discussion before the proof and derive a formula for
\(k\) by the careful analysis of the effects of the elementary transformations on the differentials. However, we choose different method, we follow the same path as in the proof
of Lemma 10.4 from \cite{OR16}. Namely, the function \(k\) is uniquely defined by the condition that \(a_{n,n-1}\theta_1+\tilde{f}\theta_2\) is \(\delta\)-invariant. Thus
we only need to compute \(\delta(\tilde{f})\).

Instead of computing \(\delta(\tilde{f})\) by brute force we use the following argument. First we present the matrix \(\tilde{X}_2\) as a sum of the upper-triangular and strictly lower-triangular
parts: \(\tilde{X}_2=\tilde{X}_{2,+}+\tilde{X}_{2,--}\).  Next we observe that \(\tilde{f}b_{n,n-1}=-\left(\Ad^{-1}_{g_{23}}\tilde{X}_{2,+}\right)_{n,n-1}\) and since \(\delta(b_{n,n-1})=0\) we have:
\[\delta(\tilde{f})=-\delta\left(\Ad^{-1}_{g_{23}}(\tilde{X}_{2,+})\right)_{n,n-1}/b_{n,n-1}.\]

On the other hand \(\Ad^{-1}_{g_{23}}(\tilde{X}_2)\) is \(\delta\)-invariant thus get
\[\delta(\tilde{f})=\delta\left(\Ad^{-1}(\tilde{X}_{2,--})\right)_{n,n-1}/b_{n,n-1}.\]

A direct computation shows that
\(\Ad^{-1}_{g_{23}}(\tilde{X}_{2,--})_{n,n-1}=b_{n-1,n-1}^2\tilde{x}_{n,n-1}/\Delta_b\) and since \(\tilde{x}_{n,n-1}\) is \(\delta\)-invariant while \(\delta b_{n-1,n-1}=b_{n,n-1}\), we obtain
\[\delta(\tilde{f})=2\tilde{x}_{n,n-1}b_{n-1,n-1}/\Delta_b.\]

Modulo, relations from \(I_g\) the matrix element \((a^{-1})_{nk}\), \(k<n\) is divisible by \(a_{n,n-1}\): \((a^{-1})_{nk}=(-1)^{k+n}a_{n,n-1}\det(M_{n,n-1}^{n,k}(a))/\Delta_a\) where
\(M_{ij}^{kl}(a)\) is the minor of \(a\) obtained by removing \(i,j\)-th columns and \(k,l\)-th rows. By putting all formulas together we finally obtain a formula for \(k\):
\[k=-2b_{n-1,n-1}\Delta_b^{-1}\left((a^{-1})_{nn}x_{nn}+\sum_{k=1}^{n-1}(-1)^{k+n}\det(M_{n,n-1}^{n,k}(a))/\Delta_a\sum_{l=k}^nx_{kl}a_{l,n-1}\right).\]

Now recall the action of torus \(T^{(2)}=(\CC)^2\subset B_2\subset B\) on \(a_{n,n-1}\) and \(\tilde{f}\) has weights \(\ep_{n}\) and \(\ep_{n-1}\) and respectively the weights of \(\theta_1,\theta_2\)
are \(-\ep_{n}\) and \(-\ep_{n-1}\). Thus \(T^{(2)}\)-invariant part of the complex \(\CE_{\frn_2}(C''_{12})\) is of the shape

$$
\begin{tikzcd}[row sep=scriptsize, column sep=scriptsize]
&\left[ \theta_1;0\right]\arrow[dl] \arrow[rr,shift left=0.5ex] \arrow[dd, dashed, two heads] & & \left[1;\ep_{n-1}\right] \arrow[dl] \arrow[dd, dashed,blue]\arrow[ll,shift left=0.5ex,blue] \\
\left[\theta_1\theta_2;-\ep_n\right]\arrow[ur, shift left=1ex,red]  \arrow[dd, dashed, red] & & \left[\theta_2;\ep_{n-1}-\ep_n\right] \arrow[ur,shift right=1ex]
\arrow[dd, dashed,red]\\
& \left[\theta_1;\ep_n-\ep_{n-1}\right] e^* \arrow[dl] \arrow[rr, shift left=0.5ex] &  & \left[1;\ep_n\right] e^* \arrow[dl]\arrow[ll, shift left=0.5ex] \\
\left[\theta_1\theta_2;\ep_{n-1}\right]e^* \arrow[ur,shift right=1ex] \arrow[rr, shift left=0.5ex] & & \left[\theta_2,0\right]e^*\arrow[ur, shift right=1ex, blue]
\arrow[ll, shift left=0.5ex,red]
\\
\arrow[to=4-3, from=1-2, dashed, crossing over,green]
\arrow[from=2-1, to=2-3, shift left=0.5ex,crossing over]
\arrow[to=2-1, from=2-3, shift left=0.5ex,crossing over]
\end{tikzcd}
$$
where the expression \([\alpha,\rho]\) stands for \(R[\rho]\alpha\), \(R=\CC[\calXr(G_n,G_{n,n-1})]\) and \(R[\rho]\) is the part of \(R\) of weight \(\rho\);
\(e^*\) is the basis of \(\frn^*=\Hom(\frn,\CC)\). In the picture the dashed arrows are the Chevalley-Eilenberg differentials.

In the tensor product \(\CC[\calXr_2]=\bar{\pi}_{13}^*(\CC[\calXr_2])]\otimes \CC[G_2]\) the first term is \(B_2\)-invariant. Hence since the vertical arrows in the diagram above
compute homology \(H^*(G_2/B,\calO(k))\) for the corresponding value of \(k\) (that could read from the
bottom of side of cube in the diagram), after contracting the vertical arrows we arrive to the diagram:
\[
  \begin{tikzcd}
    H^*(\PP^1,\calO(-2))\otimes \CC[\calXr_2]\arrow[r,shift left=0.5 ex]\arrow[rd]\arrow[d]& H^*(\PP^1,\calO(-1))\otimes \CC[\calXr_2]\arrow[l]\arrow[d]\\
    H^*(\PP^1,\calO(-1))\otimes \CC[\calXr_2]\arrow[u,shift left=0.5 ex]\arrow[r,shift left=0.5 ex]& H^*(\PP^1,\calO(0))\CC[\calXr_2]\arrow[l]\arrow[u,shift left=0.5ex].
  \end{tikzcd}
\]

Since only two vertices of the last diagram are actually non-zero we only need to compute the diagonal arrow.  The target of this arrow
 is $\textup{H}^1(\mathbb{P}^1,\calO(-2))\otimes \bar{\pi}_{13}^*(\CC[\calX_2])=\textup{H}^1_{\Lie}(\frn,R[\ep_{n-1}-\ep_n])$, hence we can
  replace the coefficients of the differential by the expressions that are homologous with respect to the differential $\delta$. Below we take advantage of this observation.
 Indeed,  note that
\[
\delta b_{n-1,n-1} = b_{n,n-1},\qquad \delta b_{n-1,n} =b_{nn},
\]
so, first, $\delta(b_{n-1,n-1}^2) = 2b_{n-1,n-1}b_{n,n-1}$, hence $b_{n-1,n-1}b_{n,n-1}$ is exact and, second, \[\delta(b_{n-1,n-1}b_{nn}) =
b_{n,n-1}b_{n-1,n} + b_{n-1,n-1}b_{nn},\] hence in view of $b_{n-1,n-1}b_{nn}-b_{n,n-1}b_{n-1,n}=\Delta_b$ we find $\exbtt\exbhh\sim \frac{1}{2}\Delta_b$. Since
\(b_{n-1,n-1}a=c\cdot(b_{n-1,n-1}b^{-1})\) we obtain:
\[2b_{n-1,n-1}a_{i,n-1}\sim  c_{i,n-1}. \]

Next let us notice that since \(a_{\bullet,i}=c_{\bullet,i}\) for \(i<n-1\) by expanding along the \(n-2\)-th column  of the determinant in the definition of \((a^{-1})_{nn}\),
we can use above homotopy equivalence we get
\[2b_{n-1,n-1} (a^{-1})_{nn}\sim \Delta_c\Delta_a^{-1}(c^{-1})_{nn}=\Delta_b(c^{-1})_{nn}.\]
We can combine the last formula with the observation that \(M_{n,n-1}^{n,k}(a)=M_{n,n-1}^{n,k}(c)\) to obtain
\[k\sim(c^{-1})_{nn}x_{nn}+\sum_{k=1}^{n-1}(-1)^{k+n}\det(M_{n,n-1}^{n,k}(c))/\Delta_c\sum_{l=k}^nx_{kl}c_{l,n-1}.\]

Next let us observe that if we collect all the terms in the last sum with \(l=n\) we obtain:
\[\sum_{k=1}^{n-1}(c^{-1})_{nk}x_{kn}=F_{n-1}(x,c)/\Delta_c+(c^{-1})_{nn}(-x_{nn}+x_{11}).\]
On the other hand if collect all the terms in the sum with \(l=s\) for \(s\ne n\) we get:
\[c_{s,n-1}\Delta^{-1}_c\sum_{k=1}^{n-1}(-1)^{k+n}\det(M_{n,n-1}^{n,k}(c))x_{ks}=(-1)^{s+n}\Delta^{-1}(c)c_{s,n-1}\det(M_{n,n-1}^{n,s}(c)x_{11})\mbox{ mod } (F_{s-1}).\]
Thus combination of the last two observations implies that modulo ideal \(I_g+(F_1,\dots,F_{n-2})\) we have the following homotopy:
\begin{multline*}
  k\sim \left((c^{-1})_{nn}x_{nn}\right)+\left(F_{n-1}(x,c)/\Delta_c+(c^{-1})_{nn}(-x_{nn}+x_{11})\right)\\
  +\left(\Delta_c^{-1}x_{11}\sum_{s=1}^{n-1}(-1)^{n+s} c_{s,n-1}\det(M_{n,n-1}^{n,s}(c))\right)=F_{n-1}(x,c)/\Delta_c^{-1}.
\end{multline*}

Finally, let us remark that \(B^2\) preserves \(F_i\) and acts linearly on the generators of \(I_g\). Thus \(\mathrm{K}^{\Wr}(\{g_{ij}\}_{i-j>1},F_1,\dots,F_{n-1})\)
is strictly \(B^2\)-equivariant.
\end{proof}

\section{Link homology computation}
\label{sec:link}

\subsection{Link homology}
\label{sec:link-homology}
In this subsection we remind our construction for link invariant from \cite{OR16} and its connection with sheaves on the nested Hilbert  scheme.

The free nested Hilbert scheme $\bHilb_{1,n}^{free}$ is a $B\times \CC^*$-quotient of the sublocus
$\widetilde{\bHilb_{1,n}^{free}}\subset \frb_n\times\frn_n\times V_n$ of the cyclic triples $\{(X,Y,v)|\CC\langle X,Y\rangle v=V_n\}$.
The usual nested Hilbert scheme $\bHilb^L_{1,n}$ is the subvariety of $\bHilb^{free}_{1,n}$, it is defined by the commutativity of
the matrices $X,Y$. Thus we have pull-back morphism:
\[ j^*_e: \MF_{B^2}(\calXr_n\times V_n,\Wr)\rightarrow \MF_{B}(\widetilde{\bHilb^{free}_{1,n}},0).\]

The complex \(\mathbb{S}_\beta:=j^*(\bclC_\beta)\) is naturally an element of the derived category \(D^{per}_{T_{sc}}(\Hilb_{1,n}^{free})\) of two-periodic complexes of coherent sheaves  on
\(\Hilb_{1,n}^{free}\). The hyper-cohomology functor \(\mathbb{H}\) is the functor \(D^{per}_{T_{sc}}(\Hilb_{1,n})\rightarrow \mathrm{Vect}_{gr}\) to the space of doubly-graded vector
spaces. There an obvious analog of vector bundle   \(\calB\) over \( \Hilb_{1,n}^{free}\) and we define
\[\mathbb{H}^k(\beta):=\mathbb{H}\left(\CE_{\frn}\left(\mathbb{S}_\beta\otimes \Lambda^k\calB\right)^T\right).\]

The main result of \cite{OR16} is the following
\begin{theorem}[\cite{OR16}]
  For any \(\beta\in \Br_n\) we have
  \begin{itemize}
  \item The cohomology    of the complex \(\mathbb{S}_\beta\) is supported on \(\Hilb_{1,n}\subset \Hilb_{1,n}^{free}\).
  \item The vector space \(\mathbb{H}^*(\beta)\) is (up to an explicit grading shift) an isotopy invariant of the closure \(L(\beta)\).
  \end{itemize}
\end{theorem}

\subsection{Koszul complex for link homology}
\label{sec:koszul-complex-link}

The virtual structure sheaf \([\calO_{Z_{1,n}^S}]^{vir}\) of the subscheme \(Z_{1,n}^S\subset \Hilb_{1,n}\) is defined as Koszul complex of
the equivariant coherent sheaves  on \(\Hilb_{1,n}^{free}\):
\[[\calO_{Z_{1,n}^{S}}]^{vir}:=\mathrm{K}(\{x_{ii}-x_{i+1,i+1}\}_{i\notin S}, \{[X,Y]_{ij}\}_{(ij)\in \tilde{S}}),\]
where \(X=(x_{ij})\), \(Y=(y_{ij})\) are the coordinates on \(\frb\) and \(\frn\) respectively and
\begin{equation}\label{eq:tilS}
  \tilde{S}=\{(ij)\}_{i-j>1},\{i+1,i\}_{i\in S}.
  \end{equation}

The zeroth homology of \([\calO_{Z_{1,n}}^S]^{vir}\)  is the structure sheaf  of \(Z_{1,n}^S\) but the complex has higher homology too. All homology are supported on
\(\Hilb_{1,n}\) and we have

\begin{proposition}
  If \(\beta=cox_S\) we have
  \[\mathbb{S}_\beta=[\calO_{Z_{1,n}^S}]^{vir}.\]
\end{proposition}
\begin{proof}
  We have shown that \(\bclC_{\beta}\) is Koszul matrix factorization with the differential
  \[D=\sum_{ij\in \tilde{S}}(g_{ij}\theta_{ij}+k_{ij}\frac{\partial}{\partial \theta_{ij}})+\sum_{i\notin S)}(F_i\theta_i+h_i\frac{\partial}{\partial\theta_u}),\]
  where \(\theta_{ij}\) and \(\theta_i\) are odd variables.
  The functions \(h_{ij}\) and \(k_i\) were not discussed previously since the Koszul matrix factorization \(\bclC_{\beta}\) is uniquely up to homotopy    determined  by the regular sequence
  \(\{g_{ij}\}_{ij\in\tilde{S}},\{F_i\}_{i\notin\tilde{S}}\). For concreteness lets construct these functions.

  For  that let us order elements of the sets \(\tilde{S}\) and \(\bar{S}=[1,n-1]\setminus S\). Then we define
  \[k_{ij}=(\Wr_{i'j'}-\Wr_{ij})/g_{ij},\quad \Wr_{ij}:=\Wr_{i'j'}|_{g_{ij}=0},\]
  where \(i'j'\) immediately precedes the element \(ij\) and if \(ij\) is the largest element of \(\tilde{S}\) then \(\Wr_{i'j'}=\Wr.\)

  Providing an explicit formulas for \(h_i\) is a bit harder but later we work with our matrix factorization in the neighborhood of \(g=1\) hence we can assume that
  \(d_i:=\det([g^{(i)}_{1\bullet},\dots,g^{(i)}_{i,\bullet}])\ne 0\) and let us also assume that the order of \(\bar{S}\) extends the natural order.
  Then from the first assumption we obtain that \(F_i/d_{i}=(x_{i+1,i+1}-x_{11})+R_i\) where \(R_i\) does not depend on variable \(x_{i+1,i+1}\). We define
  \[h_i=(\Wr_{i'}-\Wr_i)/F_i,\quad \Wr_{i}:=\Wr_{i'}|_{x_{i+1,i+1}=x_{11}+R_i},\]
  where \(i'\) immediately precedes \(i\) and if \(i\) is the largest element of \(\bar{S}\) then \(\Wr_i=\Wr_{kl}\) where \(kl\) is the smallest
  element of \(\tilde{S}\).

  Finally let us observe that from our formulas immediately follows that
  \[k_{ij}|_{g=1}=\frac{\partial \Wr}{\partial g_{ij}}|_{g=1}=[X,Y]_{ij},\quad F_i|_{g=1}=(x_{i+1,i+1}-x_{11}) .\]
  Moreover since \(\Wr\) has linear dependence on \(X\) we also get that
  \[h_i|_{g=1}=d_i^{-1}\frac{\partial\Wr}{\partial x_{i+1,i+1}}|_{g=1}=0.\]
\end{proof}

Let also remark that the dg-scheme from proposition 3.25 of \cite{GorskyNegutRasmussen16} seems to be closely related to the dg-scheme  defined by the complex \([\calO_{Z_{1,n}^S}]^{vir}\). We hope to explore this
relation in future. For more explicit connections with \cite{GorskyNegutRasmussen16} see the last section of this paper.

\subsection{Proof of theorem~\ref{thm:coxlinks}}
\label{sec:proof-theor}

Theorem 1.1.1.  from \cite{OR17} implies that
\[\mathbb{S}_{\beta\cdot \delta^k}=\mathbb{S}_{\beta}\otimes \mathcal{L}^{\vec{k}}.\]
If we apply this formula for \(\beta=\cox_S\) and combine it with the previous proposition we obtain the statement of the theorem.

\section{Explicit computations}
\label{sec:expl-comp}

In this subsection we explain how the above geometric computations translate into straight forward homological algebra. Discuss the subtleties of our construction of the knot homology
that is related to the \(t\)-grading and how this subtleties prevent us from using localization techniques in a naive way. All complexity of the situation could be seen in the
case \(n=2\) which we discuss at the end of the section.

\subsection{Details on $t$-grading}
\label{sec:details-t-grading}

Since \(\deg_t\Wr=2\) we need to explain how we need to explain how we assign the \(t\)-degree shifts in our matrix factorizations.
We fix convention for \(\mathbf{t}^k\cdot M\)
the shifted version of a module \(M\). For example for \(1\in \mathbf{t^k}\CC[\calXr_2]\) we have \(\deg_t(1)=k\).

Thus let us provide a clarification for the \(T_{sc}\)-equivariant  of the elements of our category \(\MF(\calXr_2,\Wr)\). An element
of \(\MF(\calXr_2,\Wr)\) is the two-periodic complex:
\[\dots\xrightarrow{d_{-1}} M_0\xrightarrow{d_0} M_1\xrightarrow{d_1}M_2\xrightarrow{d_2}\dots,\]
where \(M_i\) are free modules   \(M_i=M_{i+2}\), \(d_{i}=d_{i+2}\)  and differentials \(d_i\) preserve \(q\)-degree and shift \(t\)-degree by \(1\).
Let us call this property {\it degree one property}.
The category \(\MF_{B^2}(\calXr_2,\Wr)\) is the appropriate equivariant enhancement of the previous category.

For example the element \(\bar{\calC}_+\in \MF_{B^2}(\calXr_2,\Wr)\) is the two-periodic complex:
\[\dots\xrightarrow{d_1} R\xrightarrow{d_0}\mathbf{t} R\xrightarrow{d_1} R\xrightarrow{d_2}\mathbf{t}R\xrightarrow{d_3}\dots,\]
where \(R=\CC[\calX_2(G_2)]\) and \(d_i=(x_{11}-x_{22})g_{11}+x_{12}g_{21}\) for odd \(i\) and \(d_i=y_{12}g_{21}\) for even \(i\).

The elements  in the ring \(\CC[\calXr_2]\)  have even   \(t\)   degrees thus the only source for the elements of \(t\)-degree  in
\(
\mathbb{S}_\beta\) are shifts \(\mathbf{t}^k\) in our complexes. Since the the convolution needs to preserve the degree one property, we require that degree \(t\)
shifts in the Chevalley-Eilenberg complex are defined by the condition that the Chevalley-Eilenberg  differentials shift \(t\)-degree by \(1\).

As a final step of the construction of \(\mathbb{S}_\beta\) we apply the pull-back \(j_e^*\) to the complex \(\bar{\calC}_\beta\) where \(j_e\) is the
embedding of \(\widetilde{\Hilb_{1,n}^{free}}\) inside \(\frn\times \frb\). To construct \(j_e^*(\bar{\calC}_\beta)\) we need to choose an affine
cover  \(\widetilde{\Hilb_{1,n}^{free}}=\bigcup_i U_i\) by the  \(B\)-equivariant charts \(U_i\), then the pull-back \(j_e^*(\bar{\calC}_\beta)\) is
Cech complex \(\check{C}_{U_\bullet}(\bar{\calC}_\beta)\). Moreover, since we would like to preserve the degree one property, we shift \(t\)-degrees in
the Cech complex so that the Cech differentials are of \(t\)-degree \(1\).

Since we are working with \(T_{sc}\)-equivariant complexes of sheaves  on the Hilbert  scheme it is very tempting to use localization technique to obtain explicit
formulas for the super-polynomial for links. However the degree one property effectively prevents us from doing this, in most of the cases. We expand on this issue
in the section \ref{sec:two-strand-case} where we discuss the two-strand case but for now let us point out that formulas obtained by localization could only produce super-polynomial that has
only even powers of \(t\) because the elements \(\CC[\calXr_2]\) have even \(t\)-degree. On the other hand there many examples of the links with   the knot homology that
are not \(t\)-even.

To end the discussion on a positive note let us point out that HOMFLY-PT polynomial is well suited for localization technique, exactly because of the degree one property.
Let us denote by \(\chi_q(\mathbb{S})\) the \(\CC^*\)-equivariant Euler characteristics of an two-periodic complex \(\mathbb{S}\in D^b_{\CC^*}(\Hilb_{1,n}^{free})\) where
\(\CC^*\) acts with opposite weights on \(\frn\) and \(\frb\).

\begin{theorem}\cite{OR16}
  For any \(\beta\) we have
  \[P(L(\beta))=\sum_i\chi_q(\mathbb{S}_\beta\otimes \Lambda^i\calB).\]
\end{theorem}

\subsection{Conjectures for Coxeter links}
\label{sec:conj-coxet-links}
Let \(j_\Delta:\frb\times \frn\) be the \(B\)-equivariant embedding inside \(\calXr_2\) and let \(\frb_S\subset \frb\) be the subspace defined by equations
\(x_{ii}=x_{i+1,i+1}\) for \(i\notin S\).
The results of the previous section imply that we have the homotopy of the
two-periodic complexes:
\begin{equation}\label{eq:Kcox}
  j^*_\Delta(\bar{\calC}_{\cox_S})\sim K_{\cox_S}\otimes \calO_{\frb_S\times\frn}, \quad K_{cox_S}:=\bigotimes_{ij\in \tilde{S}}[R\xrightarrow{[X,Y]_{ij}} \mathbf{t}\cdot R],
\end{equation}
where \(\tilde{S}=\{j-i>1\}\cup S\) and \(R=\CC[\calXr_2]\).
The tensor product above is a restriction of the complex to the subvariety \(\frb_S\times\frn\) and to simplify notations we abbreviate the restriction  by \(K_{\cox_S}\).

Let us cover \(\widetilde{\Hilb_{1,n}^{free}}\) by the affine charts \(U_i\) then we have the following expression for the homology:
\[\mathbb{H}^m(\cox_S\cdot \delta^{\vec{k}})=\CE_{\frn}\left(\check{C}_{U_\bullet}(K_{\cox_S}\otimes\Lambda^m\calB \otimes\chi_{\vec{k}})\right)^{T},\]
where \(\chi_{\vec{k}}\) is a notation for the character of the torus \(T\).

We simplify slightly the above formula  by eliminating the Chevalley-Eilenberg complex with the following trick.
In the next section we describe affine  subspaces \(\mathbb{A}_\bullet\subset\widetilde{\Hilb^{free}_{1,n}}\) such that affine varieties \(B\mathbb{A}_\bullet\) form an affine cover
of \(\widetilde{\Hilb_{1,n}^{free}}\) and \(B\)-stabilizer is trivial. Hence if we choose \(B\mathbb{A}_\bullet\) as our Cech cover then because of the triviality of the stabilizers
the Chevalley-Eilenberg complex is acyclic on every chart and extracting its zeroth homology on the chart \(B\mathbb{A}_S\) corresponds to the restriction on the affine subvariety
\(T\mathbb{A}_S\) which we denote by \(\mathbb{T}_S\). Thus we have the following least geometry rich statement:

\begin{corollary} For any \(\vec{k}\) and \(S\) we have:
  \[\mathbb{H}^m(\cox_S\cdot \delta^{\vec{k}})=\left(\check{C}_{\mathbb{A}_\bullet}(K_{\cox_S})\otimes \Lambda^m\calB\otimes \chi_{\vec{k}}\right)^T.\]
  \end{corollary}

  Since the line bundle \(\calL^{\vec{k}}\) is very ample for sufficiently positive \(\vec{k}\) for such \(\vec{k}\) the Cech complex becomes acyclic a we have
\begin{corollary}
    For sufficiently positive \(\vec{k}\) we have:
  \[\mathbb{H}^i(\cox_S\cdot \delta^{\vec{k}})=\left(H_{\check{C}}^0(K_{\cox_S}\otimes \Lambda^i\calB\otimes \chi_{\vec{k}})\right)^{T}.\]
\end{corollary}

In the last formula  we eliminated all possible sources of odd \(t\)-degree shifts with exception of the shifts inside the complex \(K_{\cox_S}\). Thus as it is we still can not apply
localization methods to extract an explicit formulas. So let correct the complex \(K_{\cox_S}\) to make it comply with localization formula:
\begin{equation}\label{eq:coxSeven}
K_{\cox_S}^{even}=\bigotimes_{ij\in\tilde{S}}[R\xrightarrow{[X,Y]_{ij}} \mathbf{t}^2R],
\end{equation}
and let us introduce computationally friendly 'invariant':
\[\calP^{even}(L(\cox_S\cdot\delta^{\vec{k}}))=\sum_{i,j} (-1)^j\dim_{q,t}\left(H^{j}(\check{C}_{\mathbb{A}_\bullet}(K^{even}_{\cox_S}\otimes\Lambda^i\calB\otimes\chi_{\vec{k}}))\right)a^i.\]
This invariant is an equivariant Euler characteristic of the complex and in the next section we explain how one can obtain explicit localization formulas for this Euler characteristic
with localization technique.

Several recent preprints \cite{Ho17,M17} suggest that for at least for sufficiently positive  \(\vec{k}\) the sum above will non-zero terms only for \(j=0\).
Other words it is reasonable to pose:
\begin{conjecture} For sufficiently positive \(\vec{k}\) we have:
  \[\calP^{even}(L(\cox_S\cdot\delta^{\vec{k}}))=\calP(L(\cox_S\cdot\delta^{\vec{k}})).\]
\end{conjecture}

As we will see in the next subsection this conjecture is false without assumption of the positivity. It is false for very negative \(\vec{k}\).

In the last section we discuss  a stronger and more geometric version of the conjecture for \(\beta=\cox\):
\begin{conjecture}
  The higher degree hyper-cohomology of the complex \(\CE_\frn(\mathbb{S}_{\beta\cdot\delta^{\vec{k}}}\otimes \Lambda^\bullet\mathcal{B})^T\) vanish if the vector \(\vec{k}\) is sufficiently positive.
\end{conjecture}

\subsection{Two strand case}
\label{sec:two-strand-case}
In this subsection we compute homology for the links obtained by closing braids on two strands. Thus illustrate our computational technique and also one can compare
computations in section 5 of \cite{GorskyNegutRasmussen16}. The results of computation in \cite{GorskyNegutRasmussen16} and in our paper match and that provides yet another evidence for existence of a close relation between
the theory outlined in \cite{GorskyNegutRasmussen16} and our.

First let us describe the computation of the homology of \(T_{2,2n+1}=L(\sigma_1^{2n+1})\). Since \(\sigma_1=\cox_S\), \(S=\emptyset\) in this case \(\frb_S=\frn\oplus \CC\)
let us fix coordinates on it \(\frb_S=\{x_{12}E_{12}+x(E_{11}+E_{22})\}\).  Respectively we fix notation \(R=\CC[x,x_{12},y_{12}]\) for the coordinate ring on \(\frb_S\times\frn\).

The complex \(K_{\cox_S}\) in this case is just \(R\). Moreover, intersection \(\widetilde{\Hilb_{1,n}^{free}}\cap \frb_S\times \frn\) is covered with two charts \(\mathbb{A}_1=\{x_{12}\ne 0\}\),
\(\mathbb{A}_2=\{y_{12}\ne 0\}\). That is the homology \(\mathbb{H}^k(T_{2,2n+1})\) are equal to the homology of the complex:
\[ \left((R_{x_{12}}\oplus R_{y_{12}})\otimes \chi^{n-k}\rightarrow \mathbf{t}R_{x_{12}y_{12}}\otimes \chi^{n-k}\right)^{T},\]
where \(\chi:T\rightarrow\CC^*\) is the character \((\lambda,\mu)\mapsto \lambda\).

Thus the knot homology of \(T_{2,2n+1}\) is the sum of triply graded vector spaces is the tensor product of \(\CC[x]\) and the space:
\[H^0(\mathbb{P}^1,\calO(n))\oplus \mathbf{t}H^1(\mathbb{P}^1,\calO(n))\oplus\mathbf{a} H^0(\mathbb{P}^1,\mathcal{O}(n-1))\oplus\mathbf{at}H^1(\mathbb{P}^1,\calO(n-1))\]
shifted by \((\mathbf{a/t})^n\).
We can compute the super-polynomial we just need the formula  for the dimensions of the homology of the line bundles:
\[\dim_{q,t}(H^0(\mathbb{P}^1,\calO(n)))=\sum_{i=0}^n q^{2i}(t/q)^{2n-2i},\quad \dim_{q,t}(H^1(\mathbb{P}^1,\calO(n)))=\sum^{-n-2}_{i=0}(q)^{2i}(t/q)^{-2n-2i-4}.\]

The case of the torus link \(T_{2,2n}\) is more involved. Since \(S=\{1\}\) in this case \(\frb_S=\frb\). Let us denote by \(R\) the ring of functions on \(\frb\times\frn\):
\(R=\CC[x_{+},x_{-},x_{12},y_{12}]\) where \(x_+=x_{11}+x_{22}\), \(x_-=x_{11}-x_{22}\). In these notations we have
\[K_{\cox_S}=[R\xrightarrow{y_{12}x_-}\mathbf{t}R].\]

The Cech cover in this case is basically the same as in the previous case: \(\mathbb{A}_1=\{x_{12}\ne 0\}\) and \(\mathbb{A}_2=\{y_{12}\ne 0\}\). Thus the homology of the
torus link \(T_{2,2n}\) is the sum of vector spaces \(\mathbb{H}^0\oplus \mathbf{a}\mathbb{H}^1\) shifted by \((\mathbf{a/t})^n\) where
\(\mathbb{H}^i\) is homology of the complex:
\begin{equation}\label{eq:2str}
\begin{tikzcd}
\mathbf{t}  R_{x_{12}y_{12}}[n-i]\arrow[r,"y_{12}x_-"] &\mathbf{t}^2R_{x_{12}y_{12}}[n-i+1]\\
R_{y_{12}}[n-i]\oplus R_{x_{12}}[n-i]\arrow[r,"y_{12}x_-"]\arrow[u]&\mathbf{t}R_{y_{12}}[n-i-1]\oplus\mathbf{t} R_{x_{12}}[n-i-1]\arrow[u]
\end{tikzcd},
\end{equation}
where \(R[m]\) stands for the degree \(m\) part of the ring \(R\) with degrees of the generators are
\[\deg x_{12}=\deg y_{12}=1,\quad \deg_{x_-}=\deg_{x_+}=0.\]

The complex above is the tensor product of \(\CC[x_+]\) and the complex with \(x_+\) set to zero.
Thus to make our computations easier we work modulo ideal \((x_+)\), \(R'=R/(x_+)\)

 Geometrically the homology of the last complex could interpreted as homology of line bundle \(\calO(n-i)\) on the union of an projective line and an affine line that intersect
 transversally at one point. But for illustration of our methods we proceed algebraically.

 First let us observe that the horizontal differential is injective and we can contract the complex in this direction. For that we need to describe the cokernel of the map.
 Since  we have:
 \[R'_{y_{12}}[m]=\CC[(\frac{x_{12}}{y_{12}}),x_-] y_{12}^m,\]
 \[R'_{x_{12}}[m]=\CC[(\frac{y_{12}}{x_{12}}),x_-]x_{12}^m.\]
  the cokernel of the map on the on \(R'_{y_{12}}[m]\) is \(\CC[\frac{x_{12}}{y_{12}}]y_{12}^m\) and the cokernel on
 \(R'_{x_{12}}[m]\) is the sum
 \[\CC[(\frac{y_{12}}{x_{12}})]x_{12}^m\oplus x_-\CC[x_-]x_{12}^m.\]
 Finally \(R'_{x_{12}y_{12}}[m]=\CC[x_-,(\frac{x_{12}}{y_{12}})^{\pm 1}]x_{12}^m\) and the cokernel of the map on this space is
\(\CC[(\frac{x_{12}}{y_{12}})^{\pm 1}]y_{12}^m\)
 There is the induced Cech differential \(d_C\) on the cokernels
 \[\CC[\frac{x_{12}}{y_{12}}]y_{12}^m\oplus \CC[(\frac{y_{12}}{x_{12}})]x_{12}^m\oplus x_-\CC[x_-]x_{12}^m\xrightarrow{d_C}
   \CC[(\frac{x_{12}}{y_{12}})^{\pm 1}]y_{12}^m.\]
If \(m\ge 0\) this induced differential is surjective and the kernel spanned by
 \[\langle y_{12}^m, x_{12}y_{12}^{m-1},\dots,x_{12}^m\rangle\oplus x_-\CC[x_-]x_{12}^m.\]
 Let us denote the last vector space by \(V_m\).

 On other hand if \(m\) is negative then kernel and cokernel of the induced differentials are the vector spaces:
 \[x_-\CC[x_-]x_{12}^m,\quad \langle y_{12}^{-m-2},y_{12}^{-m-1}x_{12},\dots,x_{12}^{-m-2}\rangle.\]
 Let us denote the first vector space \(V'_m\) and \(V''_m\).

Thus for \(n\ge 0\)
  the knot homology of \(T_{2,2n}\) is triply graded vector space:
 \[\mathbf{a/t}^n\cdot(\mathbf{t}V_n\oplus\mathbf{at}V_{n-1})\otimes\CC[x_+],\]
 and for negative \(n\) the knot homology of \(T_{2,2n}\) is the vector space:
 \[\mathbf{a/t}^n\cdot(\mathbf{t} V'_n\oplus\mathbf{t}^2V''_n\oplus\mathbf{at}V'_{n-1}\oplus\mathbf{at^2}V''_{n-1})\otimes \CC[x_+].\]
 To convert the last formula  into super-polynomial we only need to remember: \[\deg_{q,t}x_{12}=\deg_{q,t}x_-=\deg_{q,t}x_+=q^2,\quad \deg_{q,t}y_{12}=t^2/q^2.\]

 We would like to point out that case of the links \(T_{2,2n}\) is more complex than the case of the knots \(T_{2,2n+1}\). For example in case of knots elements of
 knot homology of \(T_{2,2n+1}\) for any \(n\) have the same parity of \(t\)-degree. It is no longer true for links, the homology of \(T_{2,2n}\) for negative \(n\)
 contains elements of odd and even \(t\)-degree. Thus it seems to be very unlikely, there is some localization type formula that produces the super-polynomial
 of \(T_{2,2n}\) for (very) negative \(n\).

\section{Localization and explicit formulas for homology}
\label{sec:local-expl-form}

In this section we present an explicit formulas for the graded dimension of the homology of the Coxeter links under assumption that the corresponding braid is sufficiently positive.
First we discuss the geometry of \(\Hilb^{free}_{1,n}\) since that is the space where we perform our localization computation.

\subsection{Local charts}
\label{sec:local-charts}

It is shown in \cite{OR16} that the free Hilbert  scheme  \(\Hilb_{1,n}^{free}\) could be covered with affine charts. In this subsection we remind this construction.
First, we describe the combinatorial data used for labeling of the charts.

Let us denote by \(NS_n\) the set of the nested pairs of sets with the following properties. An element \(\bfS\in NS_n\) is a pair of nested sets:
\[\bfS_x^1\supset \bfS_x^2\supset\dots\supset \bfS_x^{n-1}\supset \bfS_x^n=\emptyset,\]
\[\bfS_y^1\supset \bfS_y^2\supset\dots\supset \bfS_y^{n-1}\supset \bfS_y^n=\emptyset,\]
such that
\[ \bfS_x^k,\bfS_y^k\subset \{ k+1,\dots,n\},\quad |\bfS_x^i|+|\bfS_y^i|=n-i.\]

Let us define the sets  of pivots of \(\bfS\) as sets \(P_x(\bfS),P_y(\bfS)\) consisting of the pairs
\[ P_x(\bfS)=\{(ij)|j\in \bfS_x^i\setminus \bfS_x^{i+1}\},\quad P_y(\bfS)=\{(ij)|j\in \bfS_y^i\setminus \bfS_y^{i+1}\}.\]
To an element \(\bfS\in NS_n\) we attach the following affine space \(\mathbb{A}_\bfS\subset \frn\times \frn\):
\[(X,Y)\in\mathbb{A}_\bfS,\mbox{ if } x_{ij}=1, ij\in P_x(\bfS)\quad y_{ij}=1, ij\in P_y(\bfS) \mbox{ and }\]
\[x_{i,j}=0,\mbox{ if } j\in \bfS_x^i,\quad y_{i,j}=0,\mbox{ if } j\in \bfS_y^i.\]
For a given \(\bfS\) we denote by \(N_x(\bfS)\) and \(N_y(\bfS)\) the indices \((ij)\) such that \(x_{ij}\) respectively \(y_{ij}\) such that the corresponding entries are not constant on
\(\mathbb{A}_\bfS\). From the construction we see that \(|N(\bfS)|=n(n-1)/2.\)

Let us denote by \(\frh\) the subspace of the diagonal matrices inside \(\frb\). The sum \(\frh+\mathbb{A}_\bfS\) is affine subspace inside \(\frb\times\frn\) and we show in \cite{OR16}:
\begin{proposition} The space \(\widetilde{\Hilb_{1,n}^{free}}\subset \frb\times\frn\) is covered by the orbits affine spaces \(B(\frh+\mathbb{A}_\bfS)\), \(\bfS\in NS_n\).
 Moreover, the points in \(\frh+\mathbb{A}_\bfS\) have trivial stabilizers.
 \end{proposition}

 Thus the proposition implies that the affine subspaces \(\frh+\mathbb{A}_\bfS\) provide an affine cover for the quotient \(\Hilb_{1,n}^{free}\). Our system for labeling of the charts
 might look a bit artificial for people studying Hilbert  schemes so let us introduce an equivalent but somewhat more familiar system.

 \subsection{Combinatorics of the cover}
\label{sec:combinatorics-cover}

Also it is probably a good place to enrich our notations to make them more compatible with the notations in \cite{GorskyNegutRasmussen16}. The free Hilbert  scheme  has a natural map \(\rho:
 \Hilb_{1,n}^{free}\rightarrow \frh\) given by the eigenvalues of the first matrix. Respectively, we define \(\Hilb_{1,n}^{free}(Z)\) to be the pre-image \(\rho^{-1}(Z)\).

 Now recall that another definition of the free Hilbert  scheme  as the space of the
 nested chains of the left ideals: \[\Hilb_{1,n}^{free}=\{I_n\subset\dots\subset I_1\subset I_0=\CC\langle X,Y\rangle|\CC\langle X,Y\rangle/I_i=\CC^i\rangle\}.\]
 Given a sequence of non-commutative monomials \(\vec{m}=(m_1,\dots,m_n)\) we define the following sublocus of the free Hilbert  scheme
 \[\mathbb{A}_{\vec{m}}=\{I_\bullet|\CC\langle X,Y\rangle/I_k=\langle m_1,\dots,m_k\rangle\}.\]

 Now let us explain how one could produce a vector of monomials \(\vec{m}(\bfS)\) from the element of \(\bfS\in NS_n\). Essentially, we just retrace the definition of the free Hilbert  scheme.
 We construct the vector inductively starting with \(m_1(\bfS)\) which is \(X\) if \((n-1,n)\in P_x(\bfS)\) and it is \(Y\) if \((n-1,n)\in P_y(\bfS)\). The inductive step is the following:
 \[m_k(\bfS)=
   \begin{cases}
     Xm_{n-j}(\bfS),\mbox{ if } &(k,j)\in P_x(\bfS)\\
     Ym_{n-j}(\bfS),\mbox{ if } &(k,j)\in P_y(\bfS).
   \end{cases}
\]

In the case of the usual nested Hilbert  scheme  it is convenient to label the torus fixed points by the standard Young tableaux (SYT).
By analogy with the commutative case we also introduce an analog of the SYT for non-commutative case. The generalized SYT, abbreviated \(GYT_n\),  
are labeling  \(L\) of \(\ZZ_{\geq 0}\times\ZZ_{\geq 0}\) by the subsets of \([1,n]\) such that every element appears once in the labeling sets.
That is  an element of \(GYT_n\) is a map \(L:\ZZ_{\geq 0}\times\ZZ_{\geq 0}\rightarrow \mbox{ subsets of } [1,n]\) with above mentioned properties.

It is natural to think about the labels as the labels on \(1\times 1\) squares that pave the first quadrant. We also require that the set of squares with
non-empty labeling is connected, other words all our generalized tableaux are connected. The standard Young tableaux are examples of generalized YT but obviously there are
many GYT which are not SYT.

There is a natural map \(GYT: NS_n\rightarrow GYT_n\) that could be described by the condition  \(k\in L(GYT(\bfS))(ij)\) if \(\deg_X(m_k(\bfS))=i\) and \(\deg_Y(m_k(\bfS))=j\).
Since the non-commutative Hilbert  scheme  contains the commutative one the image of the above map contains the set \(SYT_n\). But we do not understand
the combinatorics well. For example we do not understand the image of this map, the answer to following question is probably known to the experts:

{\bf Question:} What is the image of the map \(NS_n\rightarrow GYT_n\)?  Is this map injective?

We checked the injectivity for small \(n\) on computer. Let us also give a few examples of GYT's that are not SYT and appear in the image:
\[\begin{matrix}
  1&2& \\3&4&5\\6&&
\end{matrix}\quad\quad\quad
\begin{matrix}
  1&2\\4&*
\end{matrix}
\quad\quad\quad
\begin{matrix}
  1&2&6\\3&&7\\4&5&
\end{matrix}
\]
where \(*=\{3,5\}\).

Finally, let us observe that size of the set \(NS_n\) is \(n!\) and we expect that that there is a natural correspondence between this set
and permutations \(\mathfrak{S}_n\) of \([1,n]\). On other hand the RS algorithm assigns to an element of \(\mathfrak{S}_n\) a pair of SYT of the same shape.
Thus we expect existence of modification of the map \(GYT\) that has as target the set of the pairs from RS algorithm.
We leave this problem for the future publications where we plan to study the connection between the geometry of the non-commutative Hilbert  scheme
and the Young projectors in \(\CC[\mathfrak{S}_n]\).

\subsection{Geometry of the torus fixed locus}
Given a element \(\bfS\in NS_n\) we denote by \(M_x(\bfS)\) and \(M_y(\bfS)\) the corresponding pair of matrices from \(\widetilde{\Hilb_{1,n}^{free}}\). The entries
\(x_{ij}\), \(ij\in N_x(\bfS)\) and \(y_{ij}\), \(ij\in N_y(\bfS)\) together with coordinates along \(\frh\) provide  local coordinates at the neighborhood of the point
\(M_x(\bfS),M_y(\bfS)\). Below we provide a formula for weights of the \(T_{sc}\)-action on these coordinates.

First let us define the pair of vectors of weights \(w_x(\bfS)\) and \(w_y(\bfS)\). We define them inductively, starting with \(w_x^n(
\bfS)=0\) and \(w_y^n(\bfS)=0\).
The inductive step is provided by
\[w_x^j(\bfS)=
  \begin{cases}
    w_x^k(\bfS)+1& \mbox{ if } (jk)\in P_x(\bfS)\\
    w_x^k(\bfS)&\mbox{ if } (jk)\in P_y(\bfS)
  \end{cases},\quad
w_y^j(\bfS)=
  \begin{cases}
    w_y^k(\bfS)+1& \mbox{ if } (jk)\in P_y(\bfS)\\
    w_y^k(\bfS)&\mbox{ if } (jk)\in P_x(\bfS)
  \end{cases}
\]

The weights above are defines in such way that
\[t^{-1}\Ad_{t_x}(X)\in \mathbb{A}_{\bfS},\quad \Ad_{t_x}(Y)\in \mathbb{A}_\bfS,\]
\[\Ad_{t_y}(Y)\in\mathbb{A}_\bfS,\quad t^{-1}\Ad_{t_y}(Y)\in \mathbb{A}_\bfS,\]
for any \((X,Y)\in \mathbb{A}_\bfS\) and \(t_x=diag(t^{w_x^1},\dots,t^{w_x^n})\),
\(t_y=diag(t^{w_y^1},\dots,t^{w_y^n})\).

From discussion it immediate that the  weights of the action are given by the formula:
\[d_x(ij)=w^i_x-w^j_x+1,\quad d_y(ij)=w_y^i-w_y^j,\quad ij\in N_x(\bfS)\]
\[d_x(ij)=w^i_x-w^j_x,\quad d_y(ij)=w_y^i-w_y^j+1,\quad ij\in N_y(\bfS).\]

Now let us write a localization formula
for \(\chi(K^{even}_{cox_S}\otimes \calL^{\vec{k}}\otimes\Lambda_a(\mathcal{B}))\).
For localization formula we need the weights of the differentials in the complex. Informally we call these weights as {\it weights of obstruction space}:
\[o_x(ij)=w^i_x-w^j_x+1,\quad o_y(ij)=w^i_y-w^j_y+1.\]

We denote by \(T_\bfS\) the tangent space at \((M_x(\bfS),M_x(\bfS))\) and by \(Ob_\bfS\) the 'obstruction' space spanned by the vectors with weights \(o(ij)\), \(i-j>0\).

Armed with the above formulas we can write the localization formula  for \(\sum_i\chi(K^{even}_{cox}\otimes \calL^{\vec{k}}\otimes \Lambda^i\calB)a^i\) as
\[\sum_{\bfS\in NS_n}Q^{\vec{k}\cdot w_x}T^{\vec{k}\cdot w_y}\Omega_\bfS(Q,T,a;\cox_S),\]
\[\Omega_\bfS(Q,T,a;\cox_S)=\frac{(1-Q)^{n-|S|}\prod_{ij\in
      \tilde{S}}(1-Q^{o_x}T^{o_y})}{\prod_{ij\in N_x(S)}(1-Q^{d_x}T^{d_y})\prod_{ij\in N_y(S)}(1-Q^{d_x}T^{d_y})}
  \prod_{i=1}^{n-1}(1-a Q^{w^i_x}T^{w^i_y})\]
where \(\tilde{S}\) is given by \eqref{eq:tilS},
\(o_x=o_x(ij),o_y=o_y(ij)\),\(d_x=d_x(ij),d_y=d_y(ij)\) and \(Q,T\) variables are related to the standard variables \(q,t\) by
\[Q=q^2,\quad T=t^2/q^2.\]

Unfortunately, the sum above is not well-defined because for some \(\bfS\) the vector \((d_x,d_y)\) vanishes. It is a manifestation of the fact that scheme  \(\left(\Hilb_{1,n}^{free}\right)^{T_{sc}}\)
is not zero-dimensional. For example the family of the matrices:
\[X=\begin{bmatrix}
 0&u&0&1\\0&0&1&0\\0&0&0&0\\0&0&0&0
\end{bmatrix},\quad
Y=
\begin{bmatrix}
  0&0&0&0\\0&0&0&0\\0&0&0&1\\0&0&0&0
\end{bmatrix}
\]
where \(u\) is any, lies inside \(\mathbb{A}_\bfS\) for \(\bfS\) with \(\bfS_x=\{4,3\}\supset\{3\}\supset\{\emptyset\}\supset\{\emptyset\}\) and
\(\bfS_y=\{4\}\supset\{4\}\supset\{4\}\supset\{\emptyset\}\). It is also fixed by the
torus \(T_{sc}\).

\begin{remark}
  As we see above the torus fixed locus is not discrete in general but we expect that the locus will of virtual dimension zero. Indeed, the computer experiment suggest that for
  any \(\bfS\in NS_n\) we have inequality:
  \[\dim \left(Ob_\bfS\right)^{T_{sc}}\geq \dim \left(T_\bfS\right)^{T_{sc}}.\]
\end{remark}

However on the commutative Hilbert  scheme  the torus fixed locus is zero-dimensional and the torus fixed points are labeled by the \(SYT_n\). Let us identify the corresponding
subset \(NS_n\):
\[ [M_x(\bfS),M_y(\bfS)]=0,\mbox{ iff } \bfS\in NS_n^{syt}.\]

We propose the following
\begin{proposition}\label{prp:even} For sufficiently positive \(\vec{k}\) we have the following localization formula for
  \[\mathcal{P}^{even}(L(\cox_S\cdot \delta^{\vec{k}}))=\sum_{\bfS\in NS_n^{syt}}Q^{\vec{k}\cdot w_x}T^{\vec{k}\cdot w_y}\Omega_\bfS(Q,T,a;cox_S).\]
\end{proposition}

\begin{proof}
  Let \(\beta=\cox\cdot \delta^{\vec{k}}\). Since the complex \(\mathbb{S}_\beta\) is supported on the commutative Hilbert  scheme, the complex \(\mathbb{S}_\beta\) is contractible in the affine neighborhood of
  \((M_x(\bfS),M_y(\bfS))\) if \(\bfS\notin SYT_n\). The union \(\Hilb'_{1,n}:=\bigcup_{\bfS\in NS_n^{syt}}\mathbb{A}_\bfS\) is open subset inside \(\Hilb_{1,n}^{free}\) and by previous remark the restriction on this
  open subset does not change the total homology. Since \(T_{sc}\)-fixed locus inside \(\Hilb'_{1,n}\)  is zero-dimensional and the formula in the proposition is the standard localization formula.
\end{proof}

The above formula  is equivalent to the formula  from  corollary 1.3 of \cite{GorskyNegutRasmussen16}: the formula in \cite{GorskyNegutRasmussen16} is also a sum over SYT's and  the
corresponding terms in \cite{GorskyNegutRasmussen16} and in our formula  coincide after  we cancel the matching factors
in the numerator and denominator.

Besides similarity to the previous conjectures there are other observations that support the conjecture. For example, it is elementary to show that
\[ [M_x(\bfS),M_y(\bfS)]=0,\mbox{ iff } \bfS\in SYT_n.\]
Hence since the complex \(\calC_\beta\) is supported on the commutative Hilbert  scheme, the complex \(\calC_\beta\) is contractible in some affine neighborhood of
\((M_x(\bfS),M_y(\bfS))\) if \(\bfS\notin SYT_n\)

\subsection{Fourth grading and localization}
\label{sec:fourth-grad-local}

In this section we provide an explanation for the even super-polynomial \(\calP^{even}\) as well as some conjectures for cases when the even super-polynomials coincide with
the usual super-polynomial. As we explain below both \(\calP^{even}\) and \(\calP\) are specializations of a conjectural richer invariant.

As it is explained \cite{OR16} and outlined in section~\ref{sec:conj-coxet-links}, for  a braid \(\beta\in \Br\) one can construct an element \(\bar{\calC}_\beta
\in \MF_{B^2}(\calXr_2,\Wr)\). To compute the triply-graded homology of the link closure \(L(\beta)\), we need to work with \(\mathbb{S}_\beta=j_e^*(\bar{\calC}_\beta)\).

Two periodic complex \(\mathbb{S}_\beta\in \MF^{str}_B(\widetilde{\Hilb_{1,n}},0)\) has differential of degree \(t\) with respect to \(T_{sc}\)-action. It is shown in \cite{OblomkovRozansky18a} that
\(\bar{\calC}_\beta\) is isomorphic to a strictly \(B^2\)-equivariant matrix  factorization, thus we can assume that \(\mathbb{S}_\beta\in \MF^{str}_B(\widetilde{\Hilb}_{1,n},0)=\mathrm{D}^{per}_{T_{sc}}(\Hilb_{1,n}^{free}).\)

The objects of the derived category \(\mathrm{D}^{per}_{T_{sc}}(\Hilb_{1,n}^{free})\) are two periodic complexes of coherent \(T_{sc}\)-equivariant sheaves with
differentials of degree \(t\) with respect to \(T_{sc}\). It is more natural to consider category \(\mathrm{D}^b_{T_{sc}}(\Hilb_{1,n}^{free})\) of bounded complexes
of \(T_{sc}\)-equivariant coherent sheaves with differentials of degree \(t\). There is a folding functor that relates these categories:
\[\mathrm{Fold}:\quad \mathrm{D}^b_{T_{sc}}(\Hilb_{1,n}^{free})\to \mathrm{D}^{per}_{T_{sc}}(\Hilb_{1,n}^{free}),\quad \calC\mapsto \oplus_{n\in \ZZ} \calC[2n],\]
where \([n]\) is a notation for the homological shift and \(\calC\) is a complex of locally-free sheaves.

Clearly, not all objects in \(\mathrm{D}^{per}_{T_{sc}}(\Hilb_{1,n}^{free})\) are foldings of bounded complexes. However, as explained in the previous sections
\(\mathbb{S}_{\cox_S}=\mathrm{Fold}(K_{\cox_S})\), where treat formula \eqref{eq:Kcox} for \(K_{\cox_S}\) is interpreted as a tensor product of bounded complexes.
Let us call the two-periodic complexes in the image of \(\mathrm{Fold}\) {\it unrollable}.

It is an interesting question, for which \(\beta\in \Br_n\) the corresponding  two-periodic complex
\(\mathbb{S}_\beta\) is unrollable. For example, a two-periodic complex for the half-twist on three strands \(\beta=\sigma_1\cdot\sigma_2\cdot\sigma_1\), the two-periodic complex \(\mathbb{S}_{\beta}\) does not appear to be unrollable, see \cite{OR16}.

Let us define  \(\mathrm{D}^{b}_{T_{sc}}(\Hilb_{1,n}^{free})_{even}\) to  be  a derived category of bounded complexes of \(T_{sc}\)-equivariant complexes of coherent
sheaves with \(T_{sc}\) invariant differentials. There is a shifting functor that relates the last two categories:
\[\mathrm{Sh}_{even}: \mathrm{D}^{b}_{T_{sc}}(\Hilb_{1,n}^{free})\to \mathrm{D}^{b}_{T_{sc}}(\Hilb_{1,n}^{free})_{even},\quad\oplus_i\calC_i\mapsto \oplus_i\mathbf{t}^i\cdot\calC_i,\]
where \(\calC=(\oplus\calC_i,D)\), \(D:\calC_i\to \calC_{i+1}\).

In the context of the paper, the relevant example is \(\mathrm{Sh}_{even}(K_{\cox_S})=K_{\cox_S}^{even}\) where we interpret formula \eqref{eq:coxSeven} for \(K_{\cox_S}^{even}\) as
tensor product of bounded complexes. That motivates us to define a super-polynomial of four-variables. Suppose \(\mathbb{S}_{\beta}=\mathrm{Fold}(\hat{\mathbb{S}}_\beta)\),
\(\hat{\mathbb{S}}_\beta\in \mathrm{D}^{b}_{T_{sc}}(\Hilb_{1,n}^{free})\) for some \(\beta\in \Br_n\) then we define
\[\mathfrak{P}(\beta)=\sum_{i,j} h^{j}\dim_{q,t}H^j(\mathrm{Sh}_{even}(\hat{\mathbb{S}}_\beta\otimes \Lambda^i\calB)).\]

Thus \(\mathfrak{P}(\beta)\) is a common generalization of \(\calP(L(\cox_S\cdot \delta^{\vec{k}}))\) and of \(\calP^{even}(L(\cox_S\cdot\delta^{\vec{k}}))\):
\begin{equation}\label{eq:even-4gr}
\calP(L(\cox_S\cdot\delta^{\vec{k}}))=\mathfrak{P}(\cox_S\cdot\delta^{\vec{k}})|_{h=t^{-1}},\quad \calP^{even}(L(\cox_S\cdot\delta^{\vec{k}}))=\mathfrak{P}(\cox_S\cdot\delta^{\vec{k}})|_{h=-1}\end{equation}

\begin{proposition}\label{prop:vanish}
  The following statements are equivalent
  \begin{enumerate}
  \item \(\calP^{even}(L(\beta))=\calP(L(\beta))\),
  \item \(\mathfrak{P}(\beta)=\calP(L(\beta))\),
  \item \(\mathfrak{P}(\beta)=\calP^{even}(L(\beta))\),
    \item \(H^j(\mathrm{Sh}_{even}(\hat{\mathbb{S}}_\beta\otimes \Lambda^i\calB))=0\) for \(j\ne 0\) for all \(i\).
  \end{enumerate}
\end{proposition}
\begin{proof}
  The last three conditions are formally equivalent. Also the last condition implies the first one. Let us show that the first condition implies the last one.
  Indeed,  \(\calP^{even}(L(beta))|_{t=1}=\calP(L(\beta))|_{t=1}\) implies that vanishing of \(H^j(\mathrm{Sh}_{even}(\hat{\mathbb{S}}_\beta\otimes \Lambda^i\calB))\)
  for odd \(j\). Thus both \(\calP^{even}(L(\beta))\) and \(\calP(L(\beta))\) are sum of monomials of \(q\) and \(t\) with positive coefficients.
  Hence the formula \eqref{eq:even-4gr} implies the statement.
\end{proof}

As explained above the super-polynomial \(\calP^{even}(L(\cox_S\cdot\delta^{\vec{k}}))\) can be computed by localization technique. Moreover, using different methods we show
in \cite{OblomkovRozansky18a}

\begin{proposition}\cite{OblomkovRozansky18a} For any \(\vec{k}\in \ZZ_{>0}\) such that \(k_1>k_2>\dots>k_{n-1}\) there is \(M\) such that for any
  \(m>M\) we have:
\[\calP(\delta^{\vec{k}}\cdot FT^m)=\calP^{even}(\delta^{\vec{k}}\cdot FT^m).\]  
\end{proposition}

Thus combination of the last proposition and proposition~\ref{prp:even} implies theorem~\ref{thm:localization}.

\subsection{Conjectures}
\label{sec:conjectures}
Motivated by the previous section we state some vanishing conjectures for super-polynomial \(\mathfrak{P}(\cox\cdot \delta^{\vec{k}})\).

The structure sheaf of \(\widetilde{\Hilb}_{1,n}^{free}\) twisted by \(B\)-character \(\chi\) descend to a line bundle on \(\Hilb_{1,n}^{free}\).
Let us denote this line bundle by \(\calL^{\vec{k}}\). The line bundle \(\calL^{\vec{1}}\) corresponds to \(FT\). Based on the discussion in \cite{GorskyNegutRasmussen16} and
constructions in \cite{OblomkovRozansky18a} we propose

\begin{conjecture}\label{conj:torus}
  For any \(\vec{k}\in \ZZ_{>0}^{n-1}\) such that \(k_i\ge k_{i+1}-1\), \(i=1,\dots,n-2\) there is \(M\) such that
  \(H^j(\mathrm{Sh}_{even}(\calL^{\vec{k}+r\vec{1}}\otimes\hat{\mathbb{S}}_{\cox})\otimes \Lambda^i\calB))=0,\)
  for any \(i\), \(j\ne 0\) and \(r>M\).
\end{conjecture}

As we mentioned before for any \(m,n\), \((m,n)=1\) there is \(\vec{k}\) such that
\(L(\cox\cdot \delta^{\vec{k}})=T_{m,n}\) is an \(m,n\) torus knots. Previous studies of
the homology of torus knots and the related combinatorics allow us to provide an evidence for the above conjecture.

\begin{proposition}
  Suppose \(L(\cox\cdot \delta^{\vec{k}})=T_{m,n}\) then the conjecture~\ref{conj:torus} holds.
\end{proposition}
\begin{proof}
  It was shown in \cite{OblomkovRozansky20} that for any \(\beta\in \Br_n\) \(\calP(L(\beta))\) is equal to the super-polynomial for the Khovanov-Rozansky homology.
  On the other hand the Khovanov-Rozansky super-polynomial for \(T_{m,n}\) was computed in \cite{HogancampMellit19} and it is shown in \cite{M16} that this super-polynomial is equal to the super-polynomial from  proposition~\ref{prp:even}.

  Finally, let us notice that \(L(\cox\cdot\delta^{\vec{k}+r\vec{1}})=T_{m+rn,n}\)
  Thus we have equality \[\calP(L(\cox\cdot\delta^{\vec{k}+r\vec{1}}))=\calP^{even}(L(\cox\cdot\delta^{\vec{k}+r\vec{1}}))
  \] for all \(r\ge 0\) and the statement follows from proposition~\ref{prop:vanish}.
\end{proof}


\end{document}